\documentclass[12pt]{amsart}
\usepackage{graphicx}
\usepackage{amsmath}
\usepackage{txfonts}
\usepackage{amsfonts}
\usepackage{amssymb}
\usepackage{setspace}
\usepackage{datetime}
\usepackage{color,enumitem,graphicx}
\usepackage[colorlinks=true,urlcolor=blue,
citecolor=red,linkcolor=blue,linktocpage,pdfpagelabels,
bookmarksnumbered,bookmarksopen]{hyperref}
\usepackage{geometry}
\usepackage{titlesec}

\newcommand{\mylargefont}{\fontsize{12}{16}\selectfont}

\titleformat{\section}
    [block]
    {\normalfont\Large\bfseries\centering}
    {\bfseries\thesection}
    {1em}
    {\boldmath}

\titleformat{\subsection}{\normalfont\normalsize\bfseries}{\bfseries\thesubsection}{1em}{\boldmath}
\titleformat{\subsubsection}{\normalfont\mylargefont\bfseries}{\bfseries\thesubsubsection}{1em}{\boldmath}
\allowdisplaybreaks[4]

\newtheorem{thm}{Theorem}[section]
\newtheorem{cor}[thm]{Corollary}

\newtheorem{lem}[thm]{Lemma}

\theoremstyle{definition}

\newtheorem{rem}[thm]{Remark}
\theoremstyle{conclusion}

\theoremstyle{question}
\numberwithin{equation}{section}
\newcommand{\be}{\begin{equation}}
	\newcommand{\ee}{\end{equation}}

\begin{document}

\title[Conformally invariant system with exponential growth and nonlocal nonlinear terms]{Mixed order conformally invariant system with exponential growth and nonlocal nonlinear terms in critical dimensions}

\author{Yiwu Chen, Wei Dai, Bin Huang}

\address{School of Mathematical Sciences, Beihang University (BUAA), Beijing 100191, P. R. China}
\email{chenyiwu@buaa.edu.cn}

\address{School of Mathematical Sciences, Beihang University (BUAA), Beijing 100191, P. R. China, and Key Laboratory of Mathematics, Informatics and Behavioral Semantics, Ministry of Education, Beijing 100191, P. R. China}
\email{weidai@buaa.edu.cn}

\address{School of Mathematical Sciences, Peking University, Beijing 100871, P. R. China}
\email{bhuang25@stu.pku.edu.cn}

\thanks{Yiwu Chen is supported by the NNSF of China (No. 12571113) and the Fundamental Research Funds for the Central Universities. Wei Dai is supported by the NNSF of China (No. 12222102 \& No. 12571113), the National Science and Technology Major Project (2022ZD0116401) and the Fundamental Research Funds for the Central Universities.}

\maketitle

\begin{abstract}
In this paper, under the extremely mild assumption $u(x)= O(|x|^{K})$ as $|x|\rightarrow+\infty$ for some $K\gg1$ arbitrarily large, we classify solutions of the following mixed order conformally invariant system with exponentially increasing and nonlocal nonlinearities in $\mathbb{R}^{n}$:
$$
\left\{
\begin{aligned}
(-\Delta)^{\frac{1}{2}}u & = e^{pv} \\
(-\Delta)^{\frac{n}{2}}v & = \left(\frac{1}{|x|^2}*u^2\right)u^2
\end{aligned}
\right.
\quad \text{in}\; \mathbb{R}^n,
$$
where $n=3,\,4$, $p>0$, $u\geqslant0$, $v(x)=o(|x|^2)$ as $|x|\to\infty$ and $u$ satisfies the finite total mass condition. The finite total mass condition can be deduced from either $u \in L^\frac{2n}{n-1}(\mathbb{R}^n)$ or $u \in \dot{H}^\frac{1}{2}(\mathbb{R}^n)$. This system is closely related to the conformally invariant equations $(-\Delta)^{\frac{1}{2}}u=\left(\frac{1}{|x|^{2}}*u^2\right)u$ and $(-\Delta)^{\frac{n}{2}}u=(n-1)!e^{nu}$ in $\mathbb{R}^{n}$ with $n=3,4$, which have been quite extensively studied.
\end{abstract}


\smallskip


\section{Introductions}

In this paper, we investigate the following mixed order conformally invariant system with exponential increasing and nonlocal nonlinearities in $\mathbb{R}^n$:
\begin{equation}\label{system}
\left\{
\begin{aligned}
(-\Delta)^{\frac{1}{2}}u & = e^{pv} \\
(-\Delta)^{\frac{n}{2}}v & = \left(\frac{1}{|x|^2}*u^2\right)u^2
\end{aligned}
\right.
\quad \text{in}\; \mathbb{R}^n,
\end{equation}
where $n=3$ or $4$, $p>0$, $u\geqslant0$, $v(x)=o(|x|^2)$ as $|x|\to\infty$ and $u$ satisfies the finite total mass condition
\begin{equation}\label{fm}
\int_{\mathbb{R}^n}{\left(\frac{1}{|x|^2}*u^2\right)(y)u^2(y)\mathrm{d}y} = \iint_{\mathbb{R}^n\times\mathbb{R}^n}{\frac{u^2(x)u^2(y)}{|x-y|^2}\mathrm{d}x\mathrm{d}y}<+\infty.
\end{equation}
From the Hardy-Littlewood-Sobolev inequality (see Lemma \ref{lemma Hardy-Littlewood-Sobolev Inequality}) and the Sobolev embedding inequality, one has
$$
\iint_{\mathbb{R}^n\times\mathbb{R}^n}{\frac{u^2(x)u^2(y)}{|x-y|^2}\mathrm{d}x\mathrm{d}y}\leq C\|u\|^{4}_{L^\frac{2n}{n-1}(\mathbb{R}^n)}\leq C\|u\|^{4}_{\dot{H}^\frac{1}{2}(\mathbb{R}^n)},
$$
thus the finite total mass condition can be derived from either the integrable assumption $u \in L^\frac{2n}{n-1}(\mathbb{R}^n)$ or the natural assumption for weak solution $u \in \dot{H}^\frac{1}{2}(\mathbb{R}^n)$. System \eqref{system} is invariant under the scaling $u_{\lambda}(x):=\lambda^{\frac{n-1}{2}}u(\lambda x)$ and $v_{\lambda}(x):=v(\lambda x)+\frac{n+1}{2p}\ln\lambda$ ($\forall \, \lambda>0$).

The square root of the Laplacian $(-\Delta)^{\frac{1}{2}}$ represents a special case of fractional Laplacian $(-\Delta)^{\frac{\alpha}{2}}$ with $\alpha = 1$. For any $u \in C^{[\alpha],\{\alpha\}+\varepsilon}_{loc}(\mathbb{R}^n) \cap \mathcal{L}_\alpha(\mathbb{R}^n)$ with $n \geqslant 1$, the nonlocal operator $(-\Delta)^{\frac{\alpha}{2}}$ ($0 < \alpha < 2$) is defined by (see \cite{CT,CL2,CLM,DQ,S})
\begin{equation}\label{nonlocal defn}
(-\Delta)^{\frac{\alpha}{2}}u(x) = C_{n,\alpha} \, P.V.\int_{\mathbb{R}^n}\frac{u(x)-u(y)}{|x-y|^{n+\alpha}}\mathrm{d}y
:= C_{n,\alpha}\lim_{\varepsilon\to0}\int_{|y-x|\geqslant\varepsilon}\frac{u(x)-u(y)}{|x-y|^{n+\alpha}}\mathrm{d}y,
\end{equation}
where $[\alpha]$ denotes the integer part of $\alpha$, $\{\alpha\} := \alpha-[\alpha]$, the constant
$$
C_{n,\alpha} = \left(\int_{\mathbb{R}^n} \frac{1-\cos\left(2\pi\zeta_1\right)}{|\zeta|^{n+\alpha}}\mathrm{d}\zeta\right)^{-1}
$$
and the (slowly increasing) function space
\begin{equation}\label{0-1-space}
\mathcal{L}_\alpha(\mathbb{R}^n) := \left\{u:\mathbb{R}^n\to\mathbb{R} \left| \int_{\mathbb{R}^n}\frac{|u(x)|}{1+|x|^{n+\alpha}}\mathrm{d}x < +\infty \right.\right\}.
\end{equation}

The fractional Laplacians $(-\Delta)^{\frac{\alpha}{2}}$ can be defined equivalently via Caffarelli and Silvestre's extension method (see \cite{CS,CLM}) for $u \in C^{[\alpha],\{\alpha\}+\varepsilon}_{loc}(\mathbb{R}^n) \cap \mathcal{L}_\alpha(\mathbb{R}^n)$. For instance, the square root of the Laplacian $(-\Delta)^{\frac{1}{2}}$ can be defined equivalently for any $u \in C^{1,\varepsilon}_{loc}(\mathbb{R}^n) \cap \mathcal{L}_1(\mathbb{R}^n)$ by
\begin{equation}\label{extension}
(-\Delta)^{\frac{1}{2}}u(x) := -C_n\lim_{y\to0^+}\frac{\partial U(x,y)}{\partial y} = -C_n\lim_{y\to0^+}\int_{\mathbb{R}^n}\frac{|x-\xi|^2-ny^2}{\left(|x-\xi|^2+y^2\right)^{\frac{n+3}{2}}}u(\xi)\mathrm{d}\xi,
\end{equation}
where $U(x,y)$ is the harmonic extension of $u(x)$ in $\mathbb{R}^{n+1}_+ = \{(x,y)| \, x \in \mathbb{R}^n , \, y \geqslant 0 \}$.

The fractional Laplacian $(-\Delta)^{\frac{\alpha}{2}}$ is a nonlocal integral-differential operator. It can be used to model diverse physical phenomena, such as anomalous diffusion and quasi-geostrophic flows, turbulence and water waves, molecular dynamics, and relativistic quantum mechanics of stars (see \cite{CV,Co} and the references therein). It also has various applications in conformal geometry, probability and finance (see \cite{Be,CT,CG} and the references therein). In particular, $(-\Delta)^{\frac{\alpha}{2}}$ with $0<\alpha<2$ can also be understood as the infinitesimal generator of a stable Lévy process (see \cite{Be}).

We define $(-\Delta)^{\frac{1}{2}}u$ and $(-\Delta)^{\frac{3}{2}}v := (-\Delta)(-\Delta)^{\frac{1}{2}}v = (-\Delta)^{\frac{1}{2}}(-\Delta)v$ by definition \eqref{nonlocal defn} and its equivalent definition \eqref{extension} for $u \in C^{1,\varepsilon}_{loc}(\mathbb{R}^n) \cap \mathcal{L}_1(\mathbb{R}^n)$ and $v \in C_{loc}^{3,\varepsilon}(\mathbb{R}^n)$ such that $v$ or $\Delta v \in \mathcal{L}_1(\mathbb{R}^n)$, where $\varepsilon>0$ is arbitrarily small. Due to the nonlocal virtue of $(-\Delta)^{\frac{1}{2}}$, we need the assumption $u \in C^{1,\varepsilon}_{loc}(\mathbb{R}^n)$ and $v \in C_{loc}^{3,\varepsilon}(\mathbb{R}^n)$ with arbitrarily small $\varepsilon>0$ (merely $u \in C^1$ and $v \in C^3$ are not enough) to guarantee that $(-\Delta)^{\frac{1}{2}}u \in C(\mathbb{R}^n)$ and $(-\Delta)^{\frac{3}{2}}v \in C(\mathbb{R}^n)$ (see \cite{CLM,S}).

Throughout this paper, we assume $(u,v)$ is a pair of classical solution to the $3$-D or $4$-D system \eqref{system} in the sense that $u \in C^{1,\varepsilon}_{loc}(\mathbb{R}^n) \cap \mathcal{L}_1(\mathbb{R}^n)$ with arbitrarily small $\varepsilon > 0$, $v\in C^{4}(\mathbb{R}^{n})$ if $n=4$, and $v \in C_{loc}^{3,\varepsilon}(\mathbb{R}^n)$ with arbitrarily small $\varepsilon > 0$ such that $v$ or $\Delta v \in \mathcal{L}_1(\mathbb{R}^n)$ if $n=3$. Consequently, $(-\Delta)^{\frac{1}{2}}u$ and $(-\Delta)^{\frac{n}{2}}v$ are pointwise well-defined and continuous in the whole space $\mathbb{R}^n$.

The system \eqref{system} with the finite total mass condition \eqref{fm} is closely related to the following two single conformally invariant equations:
\begin{equation}\label{Seperate Equation 1}
(-\Delta)^\frac{1}{2}{u} = \left(\frac{1}{|x|^2}*u^2\right)u \qquad \text{in} \; \mathbb{R}^n,
\end{equation}
\begin{equation}\label{Seperate Equation 2}
(-\Delta)^\frac{n}{2}{u} = (n-1)!e^{nu} \qquad \text{in} \; \mathbb{R}^n
\end{equation}
with the finite total mass condition $\int_{\mathbb{R}^n}e^{nu}\mathrm{d}x<+\infty$.

The Hartree equation \eqref{Seperate Equation 1} is a typical special form of the following physically interesting static Schr\"{o}dinger-Hartree-Maxwell type equations involving higher-order or higher-order fractional Laplacians:
\begin{equation}\label{PDES}
(-\Delta)^{s}u(x)=\left(\frac{1}{|x|^{\sigma}}\ast u^{p}\right)u^{q}(x) \qquad \text{in} \; \mathbb{R}^{n},
\end{equation}
where $n\geq1$, $0<s<\frac{n}{2}$, $0<\sigma<n$, $0<p\leq\frac{2n-\sigma}{n-2s}$ and $0<q\leq\frac{n+2s-\sigma}{n-2s}$. When $\sigma=4s$, $p=2$, $0<q\leq1$, \eqref{PDES} is called static Schr\"{o}dinger-Hartree type equations. When $\sigma=n-2s$, $p=\frac{n+2s}{n-2s}$, $0<q\leq\frac{4s}{n-2s}$, \eqref{PDES} is known as static Schr\"{o}dinger-Maxwell type equations. When $p=\frac{2n-\sigma}{n-2s}$ and $q=\frac{n+2s-\sigma}{n-2s}$, we say equation \eqref{PDES} is conformally invariant or has critical growth (on nonlinearity).

One should observe that both the fractional Laplacians and the convolution type nonlinearity are nonlocal in equations \eqref{Seperate Equation 1} and \eqref{PDES}, which are called double nonlocal problems. PDEs of type \eqref{PDES} arise in the Hartree-Fock theory of the nonlinear Schrödinger equations (see \cite{LS}). The solution $u$ to problem \eqref{PDES} is also a ground state or a stationary solution to the following dynamic Schrödinger-Hartree equation
\begin{equation}\label{Hartree}
i\partial_{t}u+(-\Delta)^{s}u=\left(\frac{1}{|x|^{\sigma}}\ast |u|^{p}\right)|u|^{q-1}u, \qquad (t,x)\in\mathbb{R}\times\mathbb{R}^{n}
\end{equation}
involving higher-order or higher-order fractional Laplacians. The higher order and fractional order Schr\"{o}dinger-Hartree equations have many interesting applications in the quantum theory of large systems of non-relativistic bosonic atoms and molecules and the theory of laser propagation in medium (see, e.g. \cite{LMZ} and the references therein). The qualitative properties of solutions to fractional order or higher order Hartree or Choquard type equations have been extensively studied, for instance, see \cite{AGSY,CD,CDZ,CW,DFHQW,DFQ,DL,DLQ,DQ,Lieb,Liu,MZ,MS} and the references therein. In particular, for the classification result for Hartree equation \eqref{Seperate Equation 1} and equations or systems involving Hartree type nonlocal nonlinearities, refer to e.g. \cite{CD,CD0,DFHQW,DFQ,DF,DLL,DL,DLQ,DQ,DQ2}.

The equation \eqref{Seperate Equation 2} is the limiting form for critical order $\alpha=n$ of the following integer-order or fractional higher-order geometrically conformally invariant equation arising from geometric analysis:
\begin{equation}\label{GPDE}
(-\Delta)^{\frac{\alpha}{2}}u = u^{\frac{n+\alpha}{n-\alpha}} \qquad \text{in} \; \mathbb{R}^n,
\end{equation}
where $n\geqslant1$ and $\alpha \in (0,n)$. We say \eqref{GPDE} has subcritical or super-critical order if $\alpha<n$ or $\alpha>n$ respectively. In the special case $\alpha = 2<n$, equation \eqref{GPDE} is the the well-known Yamabe problem. In higher order case that $\alpha\geqslant2$ is an even integer, equation \eqref{GPDE} arises from the conformal metric problems, prescribing $Q$-curvature problems, conformally covariant Paneitz operators and GJMS operators and so on ... (see e.g. \cite{CY,CL,CL1,CL2,GJMS,Juhl,Lin,Li,P,WX,Xu,Z} and the references therein). In the fractional order or fractional higher order case that $\alpha \in (0,n) \backslash 2\mathbb{N}$, conformally invariant equation \eqref{GPDE} is closely related to the fractional $Q$-curvature problems and the study of fractional conformally covariant Paneitz and GJMS operators and so on ... (cf. e.g. \cite{CC,CG,JLX1} and the references therein).

The quantitative and qualitative properties of solutions to conformally invariant equations \eqref{GPDE} have been extensively studied. In the subcritical order case $\alpha \in (0,n)$, for classification results of positive classical solutions to equation \eqref{GPDE}, please see Gidas, Ni and Nirenberg \cite{GNN1} and Caffarelli, Gidas and Spruck \cite{CGS} for $\alpha = 2$, Lin \cite{Lin} for $\alpha = 4$, Wei and Xu \cite{WX} for even integer $\alpha \in (0,n)$, Chen, Li and Li \cite{CLL}, Chen, Li and Zhang \cite{CLZ} and Jin, Li and Xiong \cite{JLX} for $0<\alpha<2$, Dai and Qin \cite{DQ} for $\alpha = 3$, and Cao, Dai and Qin \cite{CDQ0} for any real number $\alpha \in (0,n)$. In \cite{CLO}, by developing the method of moving planes in integral forms, Chen, Li and Ou classified all the positive $L^{\frac{2n}{n-\alpha}}_{loc}$-solutions to the equivalent integral equation of the PDE \eqref{GPDE} for general $\alpha \in (0,n)$, as a consequence, they obtained the classification results for positive weak solutions to PDE \eqref{GPDE}. In the super-critical order cases, for classification results of positive classical solutions to equation \eqref{GPDE} and related IE with negative exponents, please refer to \cite{DFu,DHL,Li,N,Xu} and the references therein.

The case $n = \alpha$ is called the limiting case (or the critical order case), which is more technically difficult since the fundamental solution $c_n\ln\frac{1}{|x|}$ for $(-\Delta)^{\frac{n}{2}}$ tends to $-\infty$ as $|x|\to+\infty$ and changes signs. When $n = \alpha = 2$, by using the method of moving planes, Chen and Li \cite{CL} classified all the $C^2$ smooth solutions with finite total curvature of the Liouville equation
\begin{equation}\label{0-1}
\left\{
\begin{aligned}
& -\Delta u(x) = e^{2u(x)}, \qquad x \in \mathbb{R}^2, \\
& \int_{\mathbb{R}^2}e^{2u(x)}\mathrm{d}x<+\infty.
\end{aligned}
\right.
\end{equation}
They proved that there exists some point $x_{0} \in \mathbb{R}^2$ and some $\lambda>0$ such that
$$
u(x) = \ln\left[\frac{2\lambda}{1+\lambda^2|x-x_{0}|^2}\right] .
$$
Equations of type \eqref{0-1} arise from a variety of situations, such as from prescribing Gaussian curvature in geometry and from combustion theory in physics. For conformally invariant systems of nonlinear PDEs of Liouville type, please see \cite{CK}.

Let $g_{\mathbf{S}^2}$ be the standard metric on the unit $2$-sphere $\mathbf{S}^2$. If we consider the conformal metric $\hat{g} := e^{2w}g_{\mathbf{S}^2}$, then the Gaussian curvature $K_{\hat{g}}$ satisfies the following PDE:
\begin{equation}\label{g1}
\Delta_{g_{\mathbf{S}^2}}w+K_{\hat{g}}e^{2w} = 1 \qquad \text{on} \; \mathbf{S}^2,
\end{equation}
where $\Delta_{g_{\mathbf{S}^2}}$ denotes the Laplace-Beltrami operator with respect to the standard metric $g_{\mathbf{S}^2}$ on the sphere $\mathbf{S}^2$. In particular, if we take $K_{\hat{g}}\equiv1$ in \eqref{system}, then $\hat{g} := e^{2w}g_{\mathbf{S}^2}$ is the pull back of the standard metric $g_{\mathbf{S}^2}$ through some conformal transformation $\phi$ (i.e., $\hat{g}$ is isometric to $g_{\mathbf{S}^2}$). Through the stereographic projection $\pi$ from $\mathbf{S}^2$ to $\mathbb{R}^2$, one can see that equation \eqref{0-1} on $\mathbb{R}^2$ is equivalent to the equation \eqref{g1} on $\mathbf{S}^2$ with $K_{\hat{g}}\equiv1$.

In general, on $(\mathbf{S}^n,g_{\mathbf{S}^n})$, if we change the standard metric $g_{\mathbf{S}^n}$ to its conformal metric $\hat{g} := e^{nw}g_{\mathbf{S}^n}$ for some smooth function $w$ on the $n$-sphere $\mathbf{S}^n$, since the GJMS operator $P_{n,g}$ is conformally covariant (c.f. \cite{GJMS,P}), it turns out that there exists some scalar curvature quantity $Q_{n,g}$ of order $n$ such that
\begin{equation}\label{g3}
-P_{n,g_{\mathbf{S}^n}}(w)+Q_{n,\hat{g}}e^{nw} = Q_{n,g_{\mathbf{S}^n}} \qquad \text{on} \; \mathbf{S}^n,
\end{equation}
where the explicit formula for $P_{n,g_{\mathbf{S}^n}}$ on $\mathbf{S}^n$ with general integer $n \in \mathbb{N}^+$ is given by (c.f. \cite{Juhl}):
\begin{equation}\label{GJMS-1,2}
P_{n,g_{\mathbf{S}^n}} =
\left\{
\begin{aligned}
& \prod_{k = 1}^{\frac{n}{2}}\left[-\Delta_{g_{\mathbf{S}^n}}+\left(\frac{n}{2}-k\right)\left(\frac{n}{2}+k-1\right)\right], & \text{if} \; n \; \text{is even}, \\
& \left[-\Delta_{g_{\mathbf{S}^n}}+\frac{(n-1)^2}{4}\right]^{\frac{1}{2}}
\prod_{k = 1}^{\frac{n-1}{2}}\left[-\Delta_{g_{\mathbf{S}^n}}+\frac{(n-1)^2}{4}-k^2\right], & \text{if} \; n \; \text{is odd}.
\end{aligned}
\right.
\end{equation}
If the metric $\hat{g}$ is isometric to the standard metric $g_{\mathbf{S}^n}$, then $Q_{n,\hat{g}} = Q_{n,g_{\mathbf{S}^n}} = (n-1)!$, and hence \eqref{g3} becomes
\begin{equation}\label{g4}
-P_{n,g_{\mathbf{S}^n}}(w)+(n-1)!e^{nw} = (n-1)! \qquad \text{on} \; \mathbf{S}^n.
\end{equation}
We can reformulate the equation \eqref{g4} on $\mathbb{R}^n$ by applying the stereographic projection. Let us denote by $\pi: \, \mathbf{S}^n\to \mathbb{R}^n$ the stereographic projection which maps the south pole on $\mathbf{S}^n$ to $\infty$. That is, for any $\zeta = (\zeta_1,\cdots,\zeta_{n+1}) \in \mathbf{S}^n\subset\mathbb{R}^{n+1}$ and $x = \pi(\zeta) = (x_1,\cdots,x_n) \in \mathbb{R}^n$, then it holds $\zeta_{k} = \frac{2x_{k}}{1+|x|^2}$ for $1\leqslant k\leqslant n$ and $\zeta_{n+1} = \frac{1-|x|^2}{1+|x|^2}$. Suppose $w$ is a smooth function on $\mathbf{S}^n$, define the function $u(x) := \phi(x)+w(\zeta)$ for any $x \in \mathbb{R}^n$, where $\zeta := \pi^{-1}(x)$ and $\phi(x) := \ln\left[\frac{2}{1+|x|^2}\right] = \ln\left|J_{\pi^{-1}}\right|$. Since the GJMS operator $P_{n,g_{\mathbf{S}^n}}$ is the pull back under $\pi$ of the operator $(-\Delta)^{\frac{n}{2}}$ on $\mathbb{R}^n$ (see Theorem 3.3 in \cite{Branson1}), $w$ satisfies the equation \eqref{g4} on $\mathbf{S}^n$ if and only if the function $u$ satisfies
\begin{equation}\label{g5}
(-\Delta)^{\frac{n}{2}}u = (n-1)!e^{nu} \qquad \text{in} \,\, \mathbb{R}^n.
\end{equation}

In \cite{CY}, for general integer $n$, Chang and Yang classified the $C^n$ smooth solutions to the critical order equations \eqref{g5} under decay conditions near infinity
\begin{equation}\label{g0}
u(x) = \ln\left[\frac{2}{1+|x|^2}\right]+w\left(\zeta(x)\right)
\end{equation}
for some smooth function $w$ defined on $\mathbf{S}^n$. When $n = \alpha = 4$, Lin \cite{Lin} proved the classification results for all the $C^4$ smooth solutions of
\begin{equation}\label{0-2}
\left\{
\begin{aligned}
&\begin{aligned}
\Delta^2u(x) = 6e^{4u(x)}, \qquad x \in \mathbb{R}^4,
\end{aligned}\\
&\begin{aligned}
\int_{\mathbb{R}^4}e^{4u(x)}\mathrm{d}x<+\infty, \quad  u(x) = o\left(|x|^2\right) \quad \text{as} \quad |x|\to+\infty.
\end{aligned}
\end{aligned}
\right.
\end{equation}
When $n = \alpha$ is an even integer, Wei and Xu \cite{WX} classified the $C^n$ smooth solutions of \eqref{g5} with finite total curvature $\int_{\mathbb{R}^n}e^{nu(x)}\mathrm{d}x<+\infty$ under the assumption $u(x) = o\left(|x|^2\right)$ as $|x|\to+\infty$. Zhu \cite{Z} classified all the classical solutions with finite total curvature of the problem
\begin{equation}\label{0-4}
\left\{
\begin{aligned}
&\begin{aligned}
(-\Delta)^{\frac{3}{2}}u(x) = 2e^{3u(x)}, \qquad x \in \mathbb{R}^3,
\end{aligned}\\
&\begin{aligned}
\int_{\mathbb{R}^3}e^{3u(x)}\mathrm{d}x<+\infty, \quad u(x) = o\left(|x|^2\right) \quad \text{as} \quad |x|\to+\infty.
\end{aligned}
\end{aligned}
\right.
\end{equation}
The equation \eqref{0-4} can also be regarded as the following system with mixed order:
\begin{equation}\label{0-4s}
\left\{
\begin{aligned}
&\begin{aligned}
(-\Delta)^{\frac{1}{2}}u(x) &= 2e^{3v(x)}, & x \in \mathbb{R}^3, \\
-\Delta v(x) &= u(x), & x \in \mathbb{R}^3,
\end{aligned}\\
&\begin{aligned}
\int_{\mathbb{R}^3}e^{3v(x)}\mathrm{d}x<+\infty, \quad v(x) = o\left(|x|^2\right) \quad \text{as} \quad |x|\to+\infty.
\end{aligned}
\end{aligned}
\right.
\end{equation}
Yu \cite{Yu} classified $(u,v)\in C^2(\mathbb{R}^4)\times C^4(\mathbb{R}^4)$ to the following conformally invariant system
\begin{equation}\label{Yu-s}
\left\{
\begin{aligned}
&\begin{aligned}
-\Delta u(x) &= e^{3v(x)}, & x \in \mathbb{R}^4, \\
\Delta^2v(x) &= u^4(x), & x \in \mathbb{R}^4, \\
u(x) &\geqslant 0, & x \in \mathbb{R}^4,
\end{aligned}\\
&\begin{aligned}
\int_{\mathbb{R}^4}u^4(x)\mathrm{d}x &< +\infty, \quad
\int_{\mathbb{R}^4}e^{3v(x)}\mathrm{d}x< +\infty,
\end{aligned}\\
&\begin{aligned}
v(x) = o\left(|x|^2\right) \quad \text{as} \quad |x|\to+\infty.
\end{aligned}
\end{aligned}
\right.
\end{equation}
Subquently, Dai and Qin \cite{DQ} classified all solutions $(u,v)$ to the following planar mixed order conformally invariant system:
\begin{equation}\label{PDE-2d}
\left\{
\begin{aligned}
&\begin{aligned}
(-\Delta)^{\frac{1}{2}}u(x) &= e^{pv(x)}, & x \in \mathbb{R}^2, \\
-\Delta v(x) &= u^4(x), & x \in \mathbb{R}^2, \\
u(x) &\geqslant 0, & x \in \mathbb{R}^2,
\end{aligned}\\
&\begin{aligned}
\int_{\mathbb{R}^2}u^4(x)\mathrm{d}x<+\infty, \quad u(x) = O\left(|x|^{K}\right) \quad \text{as} \quad |x|\to+\infty,
\end{aligned}
\end{aligned}
\right.
\end{equation}
where $(u,v)\in \left(C^{1,\varepsilon}_{loc}(\mathbb{R}^2)\cap \mathcal{L}_1(\mathbb{R}^2)\right)\times C^2(\mathbb{R}^2)$ with arbitrarily small $\varepsilon>0$, $p \in (0,+\infty)$ and $K\gg1$ is arbitrarily large. For more classification results on mixed order conformally invariant systems, please c.f. \cite{CGP,DDZ,DF,DQ2,GP,GP1,HN,Yu}. For more literatures on the quantitative and qualitative properties of solutions to fractional order or higher order conformally invariant PDE and IE problems, please refer to \cite{BKN,CT,CL,CL2,DQ,JLX1,LZ1,LZ} and the references therein. For more literatures on the classification of solutions and Liouville type theorems for various PDE and IE problems via the methods of moving planes or spheres and the method of scaling spheres, please refer to \cite{CGS,CDQ0,CY,CL,CL1,CL0,CL2,CLL,CLO,CLZ,DHL,DLQ,DQ,DQ0,GNN1,JLX,JLX1,Lin,Li,LZ1,LZ,Pa,WX,Xu,Yu,Z} and the references therein.

\bigskip

Our classification result for the $3,4$-D system \eqref{system} is the following theorem.
\begin{thm}\label{theorem Main Result}
Assume $n=3$ or $4$ and $p>0$. Let $(u,v)$ be a pair of classical solutions to the system \eqref{system} such that $u \geqslant 0$, $v(x)=o(|x|^2)$ as $|x|\rightarrow+\infty$ and $u$ satisfies the finite total mass condition \eqref{fm}. Suppose $u(x) = O(|x|^K)$ as $|x|\rightarrow+\infty$ for some $K \gg 1$ arbitrarily large. Then $(u,v)$ must take the form
\begin{equation*} u(x)=\frac{2\left(\frac{2}{p}\right)^{\frac{1}{4}}\mu}{\sqrt{\pi}\left(1+\mu^{2}|x-x_{0}|^{2}\right)},\ \,\,\,\,\, \,\,\, v(x)=\frac{2}{p}\ln\left(\frac{\frac{2}{\sqrt[4]{\pi}}\left(\frac{2}{p}\right)^{\frac{1}{8}}\mu}
{1+\mu^{2}|x-x_{0}|^{2}}\right), \qquad \text{if} \,\, n=3,
\end{equation*}
\begin{align*}\label{eq42}
u(x)=\frac{2}{\sqrt{\pi}}\left(\frac{30}{p}\right)^{\frac{1}{4}}\left(\frac{\mu}{1+\mu^{2}|x-x_{0}|^{2}}\right)^
\frac{3}{2}, \,\,\,\, \,\,\,\,\,\, v(x)=\frac{5}{2p}\ln\left(\frac{\frac{\sqrt{6}}{\sqrt[5]{\pi}}\left(\frac{5}{p}\right)^{\frac{1}{10}}\mu}
{1+\mu^{2}|x-x_{0}|^{2}}\right), \qquad \text{if} \,\, n=4
\end{align*}
for some $\mu>0$ and some $x_{0}\in \mathbb{R}^{n}$.
\end{thm}

\begin{rem}\label{rem0}
The finite total mass condition \eqref{fm} and $v(x) = o(|x|^2)$ at $\infty$ is necessary for the classification result of higher order conformally invariant equations or systems (see e.g. \cite{CY,Lin,WX,Yu,Z}). The assumption ``$u(x) = O\left(|x|^{K}\right)$ at $\infty$ for some $K\gg1$ arbitrarily large" is an extremely mild condition. In fact, the necessary condition for us to define $(-\Delta)^{\frac{1}{2}}u$ is $u \in\mathcal{L}_1(\mathbb{R}^n)$ (i.e., $\frac{u}{1+|x|^{n+1}}\in L^1(\mathbb{R}^n)$), which already indicates that $u$ grows slowly and must has strictly less than linear growth at $\infty$ in the sense of integral.
\end{rem}

We would like to mention some key ideas and main ingredients in our proof of Theorem \ref{theorem Main Result}.

One should note that the $3,4$-D system \eqref{system} has higher degree of nonlinearity and complexity than a single equation. We need to overcome the mutual restrictions between $(u,v)$ and derive the crucial asymptotic behaviors at $\infty$ and the integral representation formulae of $u$ and $v$.

First, from the finite total mass condition \eqref{fm}, we can derive the integral representation formula for $u$ (see Lemmas \ref{theorem Integral Representation for n = 3}, \ref{theorem Integral Representation for n = 4}), that is,
\begin{equation}\label{Integral Representation of u}
\begin{aligned}
u(x) &= \frac{1}{2\pi^2}\int_{\mathbb{R}^3}\frac{e^{pv(y)}}{|x-y|^2}\mathrm{d}y, & \;\text{if}\; n = 3, \\
u(x) &= \frac{1}{4\pi^2}\int_{\mathbb{R}^4}\frac{e^{pv(y)}}{|x-y|^3}\mathrm{d}y, & \;\text{if}\; n = 4,
\end{aligned}
\end{equation}
and hence $|x|^{-n+1}e^{pv} \in L^1(\mathbb{R}^n)$. Then, from the finite total mass condition \eqref{fm} and $v(x) = o(|x|^2)$ at $\infty$, we can derive the integral representation formula for $\Delta v$ (see Lemmas \ref{lemma Lowerbound of xi(x) for n = 3}, \ref{lemma Integral Representation of Delta v for n = 3}, \ref{lemma Precise Integral Representation of Delta v for n = 3} and Lemmas \ref{lemma Lowerbound of xi(x) for n = 4}, \ref{lemma Integral Representation of Delta v for n = 4}, \ref{lemma Precise Integral Representation of Delta v for n = 4}), that is,
\begin{equation}\label{e2+}
\begin{aligned}
\Delta{v}(x) &= -\frac{1}{2\pi^2}\int_{\mathbb{R}^3}{\frac{P_3(y)}{|x-y|^2}\mathrm{d}y}, & \;\text{if}\; n = 3, \\
\Delta{v}(x) &= -\frac{1}{4\pi^2}\int_{\mathbb{R}^4}{\frac{P_4(y)}{|x-y|^2}\mathrm{d}y}, & \;\text{if}\; n = 4,
\end{aligned}
\end{equation}
where we denote
$$
P_n(y) \triangleq \left(\frac{1}{|\cdot|^2}*u^2\right)(y)u^2(y) = \int_{\mathbb{R}^n}{\frac{u^2(x)u^2(y)}{|x-y|^2}\mathrm{d}x}.
$$
From the finite total mass condition \eqref{fm}, $v(x) = o(|x|^2)$ at $\infty$ and the assumption $u(x) = O\left(|x|^{K}\right)$ at $\infty$ for some $K\gg1$ arbitrarily large, by using the $\exp^{L}+L\ln{L}$ inequality from \cite{DQ2} to estimate integrals with logarithmic singularity, we get the asymptotic property
$$
\lim_{|x|\to+\infty}\frac{\zeta(x)}{\ln|x|} = -\alpha,
$$
where
$$
\begin{aligned}
\zeta(x) &:= \frac{1}{2\pi^2}\int_{\mathbb{R}^3}\ln\left[\frac{|y|}{|x-y|}\right]P_3(y)\mathrm{d}y, & \alpha &:= \frac{1}{2\pi^2}\int_{\mathbb{R}^3}P_3(x)\mathrm{d}x, & \;\text{if}\; n = 3, \\
\zeta(x) &:= \frac{1}{8\pi^2}\int_{\mathbb{R}^4}\ln\left[\frac{|y|}{|x-y|}\right]P_4(y)\mathrm{d}y, & \alpha &:= \frac{1}{8\pi^2}\int_{\mathbb{R}^4}P_4(x)\mathrm{d}x, & \;\text{if}\; n = 4.
\end{aligned}
$$
Based on these properties, by the Liouville type results in Corollary 2.10 in \cite{DQ2} derived from Lemma 3.3 in Lin \cite{Lin} (see Lemma \ref{lemma Classification of Harmonic Functions}), we can deduce from $|x|^{-n+1}e^{pv} \in L^1(\mathbb{R}^n)$ the following integral representation formula for $v$ :
\begin{equation}\label{Integral Representation of v}
\begin{aligned}
v(x) &= \frac{1}{2\pi^2}\int_{\mathbb{R}^3}{P_3(y)\ln{\left(\frac{|y|}{|x-y|}\right)}\mathrm{d}y} + \gamma_3, & \;\text{if}\; n = 3, \\
v(x) &= \frac{1}{8\pi^2}\int_{\mathbb{R}^4}{P_4(y)\ln{\left(\frac{|y|}{|x-y|}\right)}\mathrm{d}y} + \gamma_4, & \;\text{if}\; n = 4
\end{aligned}
\end{equation}
for some constant $\gamma_3,\gamma_4 \in \mathbb{R}$, and hence the crucial asymptotic behavior (see Lemmas \ref{lemma Integral Representation of v for n = 3}, \ref{lemma Integral Representation of v for n = 4})
$$
\lim_{|x|\to+\infty}\frac{v(x)}{\ln|x|} = -\alpha.
$$
As a consequence, combining with \eqref{Integral Representation of u}, we derive that
$$
\begin{aligned}
\alpha &:= \frac{1}{2\pi^2}\int_{\mathbb{R}^3}P_3(x)\mathrm{d}x\geqslant\frac{1}{p}, &  \;\text{if}\; n = 3, \\
\alpha &:= \frac{1}{8\pi^2}\int_{\mathbb{R}^4}P_4(x)\mathrm{d}x\geqslant\frac{1}{p}, &  \;\text{if}\; n = 4.
\end{aligned}
$$
Moreover, if $\alpha>\frac{n}{p}$, we can get the asymptotic behavior (see Lemmas \ref{lemma Asymptotic Property for n = 3}, \ref{lemma Asymptotic Property for n = 4})
$$
\lim_{|x|\to+\infty}|x|^{n-1}u(x) = C,
$$
where
$$
\begin{aligned}
C &:= \frac{1}{2\pi^2}\int_{\mathbb{R}^3}e^{pv(x)}\mathrm{d}x<+\infty, & \;\text{if}\; n = 3,  \\
C &:= \frac{1}{4\pi^2}\int_{\mathbb{R}^4}e^{pv(x)}\mathrm{d}x<+\infty, & \;\text{if}\; n = 4.
\end{aligned}
$$

Next, by making use of the above crucial asymptotic estimates and integral representation formulae, we can apply the method of moving spheres to the IE system for $(u,v)$ consisting of \eqref{Integral Representation of u} and \eqref{Integral Representation of v}. One should note that, if $(u, v)$ solve the system \eqref{system} with $n=3,4$ for any given $p \in (0,+\infty )$, then $\widetilde{u}:=p^{\frac{1}{4}}u$ and $\widetilde{v}:=pv+\frac{1}{4}\ln p$ solve \eqref{system} with $p = 1$. Consequently, for the sake of simplicity, we may assume that $p = 1$. For any $x \in \mathbb{R}^n$ , we first prove that, $u_{x,\lambda} \geqslant u$ and $v_{x,\lambda} \geqslant v$ in $B(x,\lambda) \backslash \{x\}$ for $\lambda \in (0,+\infty)$ sufficiently small if $\alpha < n+1$; $u_{x,\lambda} \leqslant u$ and $v_{x,\lambda} \leqslant v$ in $B(x,\lambda) \backslash \{x\}$ for $\lambda \in (0,+\infty)$ sufficiently large if $\alpha > n+1$ (see \ref{Kelvin Transformation} for definitions of the Kelvin transforms $u_{x,\lambda}$ and $v_{x,\lambda}$ ). Then, for any $\overline{x} \in \mathbb{R}^n$, we show the limiting radius $\overline{\lambda}(\overline{x}) = +\infty$ if $\alpha < n+1$ and $\overline{\lambda}(\overline{x}) = 0$ if $\alpha > n+1$ (see \ref{Definition of Supremum Limiting Radius} and \ref{Definition of Infimum Limiting Radius} for definitions of the limiting radius $\overline{\lambda}(\overline{x})$), and hence derive a contradiction from Lemma 11.2 in \cite{LZ1} (see Lemma \ref{Calculus Lemma}), the finite total mass condition and the system \eqref{system}. Finally, we must have $\alpha = n+1$ and hence $u_{x,\overline{\lambda}(x)} \equiv u$ and $v_{x,\overline{\lambda}(x)} \equiv v$ in $\mathbb{R}^n \backslash \{x\}$ for any $x \in \mathbb{R}^n$ and some $\overline{\lambda}(x) \in (0,+\infty)$ depending on $x$. As a consequence, Lemma 11.1 in \cite{LZ1} (see Lemma \ref{Calculus Lemma}) and the asymptotic properties of $(u,v)$ yield the desired classification results in Theorem \ref{theorem Main Result}.

The rest of our paper is organized as follows. In Section 2, we will give some results which are useful in our proof. In Section 3, we carry out the proof of our main result -- Theorem \ref{theorem Main Result}.

In what follows, unless otherwise stated, we will use $C$ to denote a general positive constant that may depend on $p$ and $\lambda$, and whose value may differ from line to line.

\section{Preliminaries}

In this section, we give some useful lemmas that are necessary for our proof.

\subsection{Basic Propeties of Fractional Laplacians}

We need the following Maximum Principle for fractional Laplacians to establish the integral representation for $u$.

\begin{lem}[Maximum Principle \cite{CLL,S}]\label{lemma Maximum Principle}
Let $\Omega$ be a bounded domain in $\mathbb{R}^n$, $n \geqslant 2$, and $0 < \alpha < 2$. Assume that $u \in \mathcal{L}_\alpha \cap C^{[\alpha],\{\alpha\}+\varepsilon}_{loc}$ with arbitrarily small $\varepsilon > 0$ and is l.s.c on $\overline{\Omega}$. If $(-\Delta)^{\frac{\alpha}{2}}{u} \geqslant 0$ in $\Omega$ and $u \geqslant 0$ in $\mathbb{R}^n\backslash\Omega$, then $u \geqslant 0$ in $\mathbb{R}^n$. Moreover, if $u = 0$ at some point in $\Omega$, then $u = 0$ a.e. in $\mathbb{R}^n$. These conclusions also hold for unbounded domain $\Omega$ if we assume further that
$$
\liminf_{|x| \to \infty}{u(x)} \geqslant 0.
$$
\end{lem}

The following Liouville theorem for $\alpha$-harmonic function in $\mathbb{R}^n$ with $n \geqslant 2$ played a significant role in establishing the integral representation.

\begin{lem}[Liouville Theorem \cite{BKN}]\label{lemma Liouville Theorem}
Assume $n \geqslant 2$ and $0 < \alpha < 2$. Let $u$ be a strong solution of
$$
\left\{
\begin{aligned}
& (-\Delta)^{\frac{\alpha}{2}}{u}(x) = 0, & x \in \mathbb{R}^n, \\
& u(x) \geqslant 0, & x \in \mathbb{R}^n,
\end{aligned}
\right.
$$
then $u \equiv C \geqslant 0$.
\end{lem}

For domain $\Omega \subset \mathbb{R}^n$, consider the following Dirichlet problem
$$
\left\{
\begin{aligned}
(-\Delta)^\frac{\alpha}{2}{u}(x) &= f(x), & x \in \Omega, \\
u(x) &= 0, & x \notin \Omega.
\end{aligned}
\right.
$$
The corresponding Green's function $G_\alpha(x,y)$ should satisfies
$$
\left\{
\begin{aligned}
(-\Delta)^\frac{\alpha}{2}{G_\alpha}(x,y) &= \delta(x-y), & x \in \Omega, \\
G_\alpha(x,y) &= 0, & x \notin \Omega \;\text{or}\; y \notin \Omega.
\end{aligned}
\right.
$$

We have the following explicit expression for Green's function $G_{\alpha,R}$ for Dirichlet problem of $(-\Delta)^\frac{\alpha}{2}$ on the ball $B(0,R)$.
\begin{lem}[Green's function on the ball \cite{K}]\label{lemma Green Functions}
The Green function for Dirichlet problem of $(-\Delta)^\frac{\alpha}{2}$ on the ball $B(0,R)$ is
$$
G_{\alpha,R}(x,y) \triangleq \frac{C_{n,\alpha}}{|x-y|^{n-\alpha}}\int_{0}^{\frac{t_R}{d_R}}{\frac{z^{\frac{\alpha}{2}-1}}{(1+z)^\frac{n}{2}}\mathrm{d}z}, \qquad x,y \in B(0,R),
$$
with
$$
\left\{
\begin{aligned}
d_R & = \frac{|x-y|^2}{R^2}, \\
t_R & = \left(1-\frac{|x|^2}{R^2}\right)\left(1-\frac{|y|^2}{R^2}\right),
\end{aligned}
\right.
$$
and $C_{n,\alpha} \in \mathbb{R}$ is a constant. Particularly,
$$
\begin{aligned}
C_{n,\alpha} & = \frac{1}{2\pi^2}\left(\int_{0}^{+\infty}{\frac{1}{(1+z)^\frac{3}{2}\sqrt{z}}\mathrm{d}z}\right)^{-1}, & \text{if}\; n=3,\alpha=1, \\
C_{n,\alpha} & = \frac{1}{4\pi^2}\left(\int_{0}^{+\infty}{\frac{1}{(1+z)^2\sqrt{z}}\mathrm{d}z}\right)^{-1}, & \text{if}\; n=4,\alpha=1.
\end{aligned}
$$
\end{lem}

\subsection{Other Technical Lemmas}

In order to show the precise asymptotic estimates and hence the integral representation formulae for the solution pair $(u,v)$, we need the following useful $\exp^L+L\ln{L}$ inequality from \cite{DQ2}, which has its own independent interest and can be regarded as a limiting form of Young's inequality or H\"{o}lder's inequality.
\begin{lem}[$\exp^L+L\ln{L}$ Inequality \cite{DQ2}]\label{lemma expL+LlnL Inequality}
Suppose $n \geqslant 1$ and $\Omega \subseteq \mathbb{R}^n$ is a bounded or unbounded domain. Assume $f \in \exp^L(\Omega)$ and $g \in L\ln{L}(\Omega)$, then we have $fg \in L^1(\Omega)$ and
$$
\begin{aligned}
\int_{\Omega}{|f(x)g(x)|\mathrm{d}x} & \leqslant \int_{\Omega}{(e^{|f(x)|}-|f(x)|-1)\mathrm{d}x} + \int_{\Omega}{|g(x)|\ln(|g(x)|+1)\mathrm{d}x} \\
& \triangleq \lVert{f}\rVert_{\exp^L(\Omega)} + \lVert{g}\rVert_{L\ln{L}(\Omega)},
\end{aligned}
$$
where
$$
\begin{aligned}
\exp^L(\Omega)
& \triangleq
\left\{
f
\left|
\begin{aligned}
& f: \Omega \to \mathbb{C} \; \text{measurable}, \\
& \int_{\Omega}{(e^{|f(x)|}-|f(x)|-1)\mathrm{d}x} < +\infty
\end{aligned}
\right.
\right\},
\\
L\ln{L}(\Omega)
& \triangleq
\left\{
g
\left|
\begin{aligned}
& g: \Omega \to \mathbb{C} \; \text{measurable}, \\
& \int_{\Omega}{|g(x)|\ln(|g(x)|+1)\mathrm{d}x} < +\infty
\end{aligned}
\right.
\right\}.
\end{aligned}
$$
\end{lem}

We also need the following lemma on Liouville type results on harmonic functions from \cite{Lin}.
\begin{lem}[\cite{Lin}]\label{lemma Classification of Harmonic Functions}
Assume $n \geqslant 2$. Suppose that $w$ is a harmonic function in $\mathbb{R}^n$. Then we have
\\
(i) If $w^+ = O(|x|^2)$ at $\infty$, then $w$ is a polynomial of degree at most 2.
\\
(ii) If $w^+ = o(|x|^2)$ at $\infty$, then $w$ is a polynomial of degree at most 1.
\\
(iii) If $w^+ = o(|x|)$ at $\infty$, then $w \equiv C$ in $R^N$ for some constant $C$.
\end{lem}

In order to apply the moving spheres method, we need the following Hardy-Littlewood-Sobolev inequality.
\begin{lem}[Hardy-Littlewood-Sobolev Inequality \cite{FL,Lieb}]\label{lemma Hardy-Littlewood-Sobolev Inequality}
Let $ n \geqslant 1 $, $ 0 < s < n $, and $ 1 < p < q < +\infty $ satisfy $ \frac{n}{p} = s + \frac{n}{q} $. Then there exists a constant $ C_{n,s,p,q} $ such that
$$
\left\lVert{\int_{\mathbb{R}^n}{\frac{f(y)}{|x-y|^{n-s}}\mathrm{d}y}}\right\rVert_{L^q(\mathbb{R}^n)} \leqslant C_{n,s,p,q}\lVert{f}\rVert_{L^p(\mathbb{R}^n)}.
$$
\end{lem}

Finally, in order to derive our result, we need the following calculus Lemma (see Lemmas 11.1 and 11.2 in \cite{LZ1}).

\begin{lem}[Calculus Lemma \cite{LZ1}]\label{Calculus Lemma}
Let $n \geqslant 1$, $\nu \in R$ and $u \in C^1(\mathbb{R}^n)$. For every $x \in \mathbb{R}^n$ and $\lambda \geqslant 0$, define
$$
u_{x,\lambda}(\xi) = \left(\frac{\lambda}{|\xi-x|}\right)^\nu u\left(x+\lambda^2\frac{\xi-x}{|\xi-x|^2}\right), \qquad \forall \xi \in \mathbb{R}^n\backslash\{x\}.
$$
Then, we have
\\
(i) If for every $x \in \mathbb{R}^n$, there exists a $0 < \lambda_x < +\infty$ such that
$$
u_{x,\lambda_x}(\xi) = u(\xi), \qquad \forall \xi \in \mathbb{R}^n\backslash\{x\}.
$$
then for some $C \in R$, $\mu > 0$ and $x_0 \in \mathbb{R}^n$,
$$
u(x) = C\left(\frac{\mu}{1+\mu^2|x-x_0|^2}\right)^\frac{\nu}{2}.
$$
\\
(ii) If for every $x \in \mathbb{R}^n$ and any $0 < \lambda < +\infty$ such that
$$
u_{x,\lambda}(\xi) \geqslant u(\xi), \qquad \forall \xi \in B(x,\lambda) \backslash \{x\},
$$
then $u \equiv C$ for some constant $C \in \mathbb{R}$.
\end{lem}

\begin{rem}\label{rem11}
In Lemma 11.1 and Lemma 11.2 of \cite{LZ1}, Li and Zhang have proved Lemma \ref{Calculus Lemma} for $\nu>0$. Nevertheless, their methods can also be applied to show Lemma \ref{Calculus Lemma} in the cases $\nu\leqslant0$, see \cite{Li,LZ,Xu}.
\end{rem}

\section{Proof of Theorem \ref{theorem Main Result}}

\subsection{Integral Representations and Crucial Asymptotic Property}

From now on, we set
$$
P_n(y) \triangleq \left(\frac{1}{|\cdot|^2}*u^2\right)(y)u^2(y) = \int_{\mathbb{R}^n}{\frac{u^2(x)u^2(y)}{|x-y|^2}\mathrm{d}x}.
$$

We will prove the following integral representation formulae for solutions.
\begin{thm}[Integral Representation for $n = 3$]\label{theorem Integral Representation for n = 3}
Assume $n=3$ and $p>0$. Let $(u,v)$ be a pair of classical solutions to the system \eqref{system} such that $u \geqslant 0$, $v(x)=o(|x|^2)$ as $|x|\rightarrow+\infty$ and $u$ satisfies the finite total mass condition \eqref{fm}. Suppose $u(x) = O(|x|^K)$ as $|x|\rightarrow+\infty$ for some $K \gg 1$ arbitrarily large. Then we have the integral representation
$$
\left\{
\begin{aligned}
& u(x) = \frac{1}{2\pi^2}\int_{\mathbb{R}^3}{\frac{e^{pv(y)}}{|x-y|^2}\mathrm{d}y}, \\
& v(x) = \frac{1}{2\pi^2}\int_{\mathbb{R}^3}{P_3(y)\ln{\left(\frac{|y|}{|x-y|}\right)}\mathrm{d}y} + \gamma_3,
\end{aligned}
\right.
$$
where $\gamma_3 \in \mathbb{R}$ is a constant.
\end{thm}

\begin{thm}[Integral Representation for $n = 4$]\label{theorem Integral Representation for n = 4}
Assume $n=4$ and $p>0$. Let $(u,v)$ be a pair of classical solutions to the system \eqref{system} such that $u \geqslant 0$, $v(x)=o(|x|^2)$ as $|x|\rightarrow+\infty$ and $u$ satisfies the finite total mass condition \eqref{fm}. Suppose $u(x) = O(|x|^K)$ as $|x|\rightarrow+\infty$ for some $K \gg 1$ arbitrarily large. Then we have the integral representation
$$
\left\{
\begin{aligned}
& u(x) = \frac{1}{4\pi^2}\int_{\mathbb{R}^4}{\frac{e^{pv(y)}}{|x-y|^3}\mathrm{d}y}, \\
& v(x) = \frac{1}{8\pi^2}\int_{\mathbb{R}^4}{P_4(y)\ln{\left(\frac{|y|}{|x-y|}\right)}\mathrm{d}y} + \gamma_4,
\end{aligned}
\right.
$$
where $\gamma_4 \in \mathbb{R}$ is a constant.
\end{thm}

\subsubsection{Proof of Theorem \ref{theorem Integral Representation for n = 3}}

In this subsection, we will prove Theorem \ref{theorem Integral Representation for n = 3}.

\begin{lem}[Integral Representation of $u$]\label{lemma Integral Representation of u for n = 3}
Assume $n = 3$. Let $(u,v)$ be a pair of classical solutions to the system \eqref{system} such that $u \geqslant 0$ and satisfies \eqref{fm}. Then we have the integral representation as
$$
u(x) = \frac{1}{2\pi^2}\int_{\mathbb{R}^3}{\frac{e^{pv(y)}}{|x-y|^2}\mathrm{d}y}.
$$
In addition, $u>0$ in $\mathbb{R}^3$, $u(x) \geqslant \frac{c}{|x|^2}$ for some constant $c>0$ and $|x|$ large enough, and
$$
\int_{\mathbb{R}^3}{\frac{e^{pv(y)}}{|y|^2}\mathrm{d}y} < +\infty.
$$
\end{lem}
\begin{proof}
For arbitrary $R>0$, let
$$
u_R(x) \triangleq \int_{B(0,R)}{G_R(x,y)e^{pv(y)}\mathrm{d}y},
$$
where Green's function for $(-\Delta)^\frac{1}{2}$ on $B(0,R) \subset \mathbb{R}^3$ is given by
$$
G_R(x,y) \triangleq \frac{C}{|x-y|^2}\int_{0}^{\frac{t_R}{d_R}}{\frac{1}{(1+z)^\frac{3}{2}\sqrt{z}}\mathrm{d}z}, \qquad x,y \in B(0,R),
$$
and
$$
\left\{
\begin{aligned}
d_R & = \frac{|x-y|^2}{R^2}, \\
t_R & = \left(1-\frac{|x|^2}{R^2}\right)\left(1-\frac{|y|^2}{R^2}\right), \\
C & = \frac{1}{2\pi^2}\left(\int_{0}^{+\infty}{\frac{1}{(1+z)^\frac{3}{2}\sqrt{z}}\mathrm{d}z}\right)^{-1}.
\end{aligned}
\right.
$$
We can deduce that $u_R \in C(\mathbb{R}^3) \cap \mathcal{L}_1(\mathbb{R}^3) \cap C^{1,\varepsilon}_{loc}(B(0,R))$ solves
$$
\left\{
\begin{aligned}
& (-\Delta)^\frac{1}{2}u_R = e^{pv} \qquad \text{in} \; B(0,R), \\
& u_R = 0 \qquad\qquad\quad\;\; \text{in} \; \mathbb{R}^3 \backslash B(0,R).
\end{aligned}
\right.
$$
Using Lemma \ref{lemma Maximum Principle}, for any $R > 0$, we have
$$
u \geqslant u_R .
$$
Taking $R\to\infty$, we get
$$
u(x) \geqslant \frac{1}{2\pi^2}\int_{\mathbb{R}^3}{\frac{e^{pv(y)}}{|x-y|^2}\mathrm{d}y} \triangleq \overline{u}(x).
$$
Then by Lemma \ref{lemma Liouville Theorem}, we get
$$
u(x) = \overline{u}(x) + C_u \geqslant 0.
$$
If $u(x) \geqslant C_u > 0$, then
$$
\iint_{\mathbb{R}^3\times\mathbb{R}^3}{\frac{u^2(x)u^2(y)}{|x-y|^2}\mathrm{d}x\mathrm{d}y}\geq \iint_{\mathbb{R}^3\times\mathbb{R}^3}{\frac{C_u^4}{|x-y|^2}\mathrm{d}x\mathrm{d}y} = +\infty,
$$
which is a contradiction. Thus $C_u$ must be zero, i.e.,
$$
u(x) = \frac{1}{2\pi^2}\int_{\mathbb{R}^3}{\frac{e^{pv(y)}}{|x-y|^2}\mathrm{d}y},
$$
and hence $u>0$ in $\mathbb{R}^3$, and
$$
\int_{\mathbb{R}^3}{\frac{e^{pv(y)}}{|y|^2}\mathrm{d}y}\leq 2\pi^2 u(0)< +\infty.
$$
Moreover, for $x$ large enough,
$$
\begin{aligned}
u(x)\geqslant & \frac{1}{2\pi^2}\int_{\mathbb{R}^3}{\frac{e^{pv(y)}}{|x-y|^2}\mathrm{d}y} \\
\geqslant & \frac{1}{2\pi^2}\int_{1 \leqslant |y| < \frac{|x|}{2}}{\frac{e^{pv(y)}}{|x-y|^2}\mathrm{d}y} \\
\geqslant & \frac{2}{9\pi^2|x|^2}\int_{1 \leqslant |y| < \frac{|x|}{2}}{\frac{e^{pv(y)}}{|y|^2}\mathrm{d}y} \\
\geqslant & \frac{1}{9\pi^2|x|^2}\int_{|y| \geqslant 1}{\frac{e^{pv(y)}}{|y|^2}\mathrm{d}y} \\
\geqslant & \frac{C}{|x|^2}.
\end{aligned}
$$
\end{proof}

Due to \eqref{fm}, we can define
$$
\zeta(x) \triangleq \frac{1}{2\pi^2}\int_{\mathbb{R}^3}{P_3(y)\ln{\left(\frac{|y|}{|x-y|}\right)}\mathrm{d}y}.
$$

Next, we need to prove the integral representation formula and the asymptotic property of $v$. Before that, we need the following lemmas.

\begin{lem}\label{lemma Lowerbound of xi(x) for n = 3}
Assume $n = 3$. Let $(u,v)$ be a pair of classical solutions to the system \eqref{system} such that $u \geqslant 0$ and satisfies \eqref{fm}. Then there is a constant $C$ such that when $|x|$ is sufficiently large, one has
$$
-\zeta(x) \leqslant \alpha\ln{|x|}+C,
$$
where $\alpha \triangleq \frac{1}{2\pi^2}\int_{\mathbb{R}^3}{P_3(y)\mathrm{d}y} \in (0,+\infty)$.
\end{lem}
\begin{proof}
Fix $x$ such that $|x|$ large enough, let us denote
$$
\begin{aligned}
& S_1 = \left\{y \left| |x-y| \leqslant \frac{|x|}{2} \right.\right\}, \\
& S_2 = \left\{y \left| |x-y| \geqslant \frac{|x|}{2} \right.\right\}.
\end{aligned}
$$
It is easy to check that $\mathbb{R}^3 = S_1 \cup S_2$ and for $y \in S_1$, one has $|y| \geqslant |x|-|x-y| \geqslant \frac{|x|}{2} \geqslant |x-y|$, which indicates $\ln{\left(\frac{|x-y|}{|y|}\right)} \leqslant 0$. For $y \in S_2$ and $|y| \geqslant 2$, one has $|x-y| \leqslant |x|+|y| \leqslant |x||y|$ (note that $|x|$ is sufficiently large). Finally, for $y \in S_2$ and $|y| \leqslant 2$, one has $|x-y| \leqslant |x|$ (note that $|x|$ is sufficiently large), which also indicates $\ln{\left(\frac{|x-y|}{|y|}\right)} \leqslant 0$.
Therefore,
$$
\begin{aligned}
-\zeta(x) & = \frac{1}{2\pi^2}\int_{\mathbb{R}^3}{P_3(y)\ln{\left(\frac{|x-y|}{|y|}\right)}\mathrm{d}y} \\
& \leqslant \frac{1}{2\pi^2}\int_{|x-y| \geqslant \frac{|x|}{2}}{P_3(y)\ln{\left(\frac{|x-y|}{|y|}\right)}\mathrm{d}y} \\
& \leqslant \frac{\ln{|x|}}{2\pi^2}\int_{\{|x-y| \geqslant \frac{|x|}{2}\} \cap\{|y| \geqslant 2\}}{P_3(y)\mathrm{d}y} + C \\
& \leqslant \alpha\ln{|x|}+C.
\end{aligned}
$$
\end{proof}

Now we can prove the integral representation for $\Delta{v}$.

\begin{lem}\label{lemma Integral Representation of Delta v for n = 3}
Assume $n = 3$. Let $(u,v)$ be a pair of classical solutions to the system \eqref{system} such that $u \geqslant 0$ and satisfies \eqref{fm}. Then
$$
\Delta{v}(x) = -\frac{1}{2\pi^2}\int_{\mathbb{R}^3}{\frac{P_3(y)}{|x-y|^2}\mathrm{d}y}-C
$$
for some constant $C \geqslant 0$.
\end{lem}
\begin{proof}
Denote $m(x) = v(x)-\zeta(x)$, from \eqref{system}, we have
$$
(-\Delta)^\frac{3}{2}{m} = 0,
$$
which implies $(-\Delta)^2{m} = 0$, i.e., $\Delta{m}$ is harmonic in $\mathbb{R}^3$. From the mean value
theorem of harmonic function, for any $x_0 \in \mathbb{R}^3$ and $r > 0$, we derive
$$
\begin{aligned}
\Delta{m}(x_0) & = \frac{3}{4\pi r^2}\int_{B(x_0,r)}{\Delta{m}(y)\mathrm{d}y} \\
& = \frac{3}{4\pi r^2}\int_{\partial B(x_0,r)}{\frac{\partial m}{\partial r}\mathrm{d}\sigma}.
\end{aligned}
$$
Integrating from $0$ to $r$, we obtain
$$
\frac{r^2}{6}\Delta{m}(x_0) = \fint_{\partial B(x_0,r)}{m(y)\mathrm{d}\sigma} - m(x_0).
$$
Using the Jensen's inequality, we have
$$
\begin{aligned}
e^{\frac{r^2}{6}\Delta{m}(x_0)} & = e^{-m(x_0)}e^{\fint_{\partial B(x_0,r)}{m(y)\mathrm{d}\sigma}} \\
& \leqslant e^{-m(x_0)}\fint_{\partial B(x_0,r)}{e^{m(y)}\mathrm{d}\sigma}.
\end{aligned}
$$
By Lemma \ref{lemma Lowerbound of xi(x) for n = 3}, we have $m(x) = v(x) - \zeta(x) \leqslant v(x) + \alpha\ln{|x|}+C$. Combining this with the above inequality and $|x|^{-2}e^{pv(x)} \in L^1(\mathbb{R}^3)$ in Lemma \ref{lemma Integral Representation of u for n = 3}, we get
$$
r^{-p\alpha}e^{\frac{pr^2}{6}\Delta{m}(x_0)} \in L^1_r([1,+\infty)).
$$
Thus $\Delta{m}(x_0) \leqslant 0$ for all $x_0 \in \mathbb{R}^3$. By Liouville Theorem, $\Delta{m}(x) = -C$ in $\mathbb{R}^3$ for some constant $C \geqslant 0$. Together with $\Delta{\zeta}(x) = -\frac{1}{2\pi^2}\int_{\mathbb{R}^3}{\frac{P_3(y)}{|x-y|^2}\mathrm{d}y}$, we get $\Delta{v}(x) = -\frac{1}{2\pi^2}\int_{\mathbb{R}^3}{\frac{P_3(y)}{|x-y|^2}\mathrm{d}y} - C$.
\end{proof}

Now we need to assume $v(x) = o(|x|^2)$ at $\infty$, and get the precise integral representation formula for $\Delta{v}(x)$.

\begin{lem}\label{lemma Precise Integral Representation of Delta v for n = 3}
Assume $n = 3$. Let $(u,v)$ be a pair of classical solutions to the system \eqref{system} such that $u \geqslant 0$ and satisfies \eqref{fm}, and $v(x) = o(|x|^2)$ at $\infty$. Then we have
$$
\Delta{v}(x) = -\frac{1}{2\pi^2}\int_{\mathbb{R}^3}{\frac{P_3(y)}{|x-y|^2}\mathrm{d}y}.
$$
\end{lem}
\begin{proof}
By Lemma \ref{lemma Integral Representation of Delta v for n = 3}, we have
$$
\Delta{v}(x) = -\frac{1}{2\pi^2}\int_{\mathbb{R}^3}{\frac{P_3(y)}{|x-y|^2}\mathrm{d}y} - C.
$$
If $C > 0$, take $\varepsilon \in (0,\frac{C}{12})$ and $R_0$ large enough such that
$$
\begin{aligned}
& \Delta{v}(x) \leqslant -C < 0 \;\text{in}\; \mathbb{R}^3, \\
& v(y) + \varepsilon|y|^2 \geqslant \frac{\varepsilon}{2}|y|^2 \geqslant 0 \quad \text{for}\; |y| \geqslant R_0.
\end{aligned}
$$
For $B > 0$, define
$$
n(y) = v(y) + \varepsilon|y|^2+B(|y|^{-1}-R_0^{-1}).
$$
Thus
$$
\Delta{n}(y) = \Delta{v}(y) + 6\varepsilon < -\frac{C}{2} < 0
$$
for $|y| \geqslant R_0$. Furthermore, we get
$$
\begin{aligned}
& \lim_{|y|\to+\infty}{n(y)} = +\infty \;\text{for any}\; B > 0, \\
& \lim_{B\to+\infty}{n(y)} = -\infty \;\text{for any}\; y \in \mathbb{R}^3 \backslash \overline{B(0,R_0)}.
\end{aligned}
$$
Thus, we can take $B$ sufficiently large such that
$\displaystyle \inf_{|y| \geqslant R_0}{n(y)}$ is attained by some $y_0 \in \mathbb{R}^3$ with $|y_0| > R_0$. By applying the maximum principle to $n(y)$, we derive a contradiction. That implies $C = 0$.
\end{proof}

Finally, we can prove the integral representation formula and the asymptotic property for $v$.

\begin{lem}[Integral Representation of $v$]\label{lemma Integral Representation of v for n = 3}
Assume $n=3$ and $p>0$. Let $(u,v)$ be a pair of classical solutions to the system \eqref{system} such that $u \geqslant 0$, $v(x)=o(|x|^2)$ as $|x|\rightarrow+\infty$ and $u$ satisfies the finite total mass condition \eqref{fm}. Suppose $u(x) = O(|x|^K)$ as $|x|\rightarrow+\infty$ for some $K \gg 1$ arbitrarily large. Then we have
$$
v(x) = \zeta(x)+\gamma = \frac{1}{2\pi^2}\int_{\mathbb{R}^3}{P_3(y)\ln{\left(\frac{|y|}{|x-y|}\right)}\mathrm{d}y}+\gamma,
$$
where $\gamma \in \mathbb{R}$ is a constant. Moreover,
$$
\lim_{|x|\to\infty}{\frac{v(x)}{\ln{|x|}}} = -\alpha,
$$
where $\alpha \triangleq \frac{1}{2\pi^2}\int_{\mathbb{R}^3}{P_3(y)\mathrm{d}y} \in (0,+\infty)$.
\end{lem}
\begin{proof}
We will first prove the following asymptotic property:
$$
\lim_{|x|\to\infty}{\frac{\zeta(x)}{\ln{|x|}}} = -\alpha.
$$
To this end, we only need to show that
$$
\lim_{|x|\to\infty}{\int_{\mathbb{R}^3}{P_3(y)\frac{\ln{|x-y|}-\ln{|y|}-\ln{|x|}}{\ln{|x|}}\mathrm{d}y}} = 0.
$$
By using the $\exp^L+L\ln{L}$ inequality in Lemma \ref{lemma expL+LlnL Inequality}, we get
$$
\begin{aligned}
& \int_{B(x,1)}{P_3(y)\ln\left(\frac{1}{|x-y|}\right)\mathrm{d}y} \\
\leqslant & \int_{B(x,1)}{\frac{1}{|x-y|}\mathrm{d}y}+\int_{B(x,1)}{P_3(y)\ln\left(P_3(y)+1\right)\mathrm{d}y} \\
\leqslant & 2\pi+\left[\max_{|y-x|\leqslant1}\ln\left(P_3(y)+1\right)\right]\int_{B(x,1)}{P_3(y)\mathrm{d}y}.
\end{aligned}
$$
As a consequence, by the above inequality, the conditions \eqref{fm} and $u = O(|x|^K) $ at $\infty$ for some $K \gg 1$ arbitrarily large, we have, for any $|x| \geqslant e^2$ large enough,
$$
\begin{aligned}
& \left|\int_{\mathbb{R}^3}{P_3(y)\frac{\ln{|x-y|}-\ln{|y|}-\ln{|x|}}{\ln{|x|}}\mathrm{d}y}\right| \\
\leqslant & 3\int_{B(x,1)}{P_3(y)\mathrm{d}y}+\frac{2\pi}{\ln{|x|}}+\frac{O(2K\ln{|x|})}{\ln{|x|}}\int_{B(x,1)}{P_3(y)\mathrm{d}y} \\
& \; + \frac{\displaystyle \max_{|y|\leqslant\ln{|x|}}{\left|\ln{\frac{|x-y|}{|x|}}\right|}}{\ln{|x|}}\int_{|y|<\ln{|x|}}{P_3(y)\mathrm{d}y}+\frac{1}{\ln{|x|}}\int_{|y|<\ln{|x|}}{|\ln{|y|}|P_3(y)\mathrm{d}y} \\
& \; + \sup_{\substack{|y-x|\geqslant1 \\ |y|\geqslant\ln{|x|}}}{\frac{|\ln{|x-y|}-\ln{|y|}-\ln{|x|}|}{\ln{|x|}}}\int_{|y|\geqslant\ln{|x|}}{P_3(y)\mathrm{d}y} \\
\leqslant & o_{|x|}(1)+\frac{2\pi}{\ln{|x|}}+\frac{\ln{2}}{\ln{|x|}}\int_{\mathbb{R}^3}{P_3(y)\mathrm{d}y}+\frac{1}{\ln{|x|}}\int_{|y|<1}{\ln\left(\frac{1}{|y|}\right)P_3(y)\mathrm{d}y} \\
& \; + \frac{\ln(\ln{|x|})}{\ln{|x|}}\int_{\mathbb{R}^3}{P_3(y)\mathrm{d}y}+\left(2+\frac{\ln{2}}{\ln{|x|}}\right)\int_{|y|\geqslant\ln{|x|}}{P_3(y)\mathrm{d}y} \\
= & o_{|x|}(1),
\end{aligned}
$$
where we have used the fact $1 > \frac{1}{|x|}+\frac{1}{|y|} \geqslant \frac{|x-y|}{|x|\cdot|y|} \geqslant \frac{1}{2|x|^2}$ for any $|y-x| \geqslant 1$ and $|y| \geqslant \ln{|x|}$. By letting $|x| \to +\infty$, we obtain
$$
\lim_{|x|\to\infty}{\int_{\mathbb{R}^3}{P_3(y)\frac{\ln{|x-y|}-\ln{|y|}-\ln{|x|}}{\ln{|x|}}\mathrm{d}y}} = 0.
$$
By Lemma \ref{lemma Precise Integral Representation of Delta v for n = 3} and $\Delta{\zeta}(x) = -\frac{1}{2\pi^2}\int_{\mathbb{R}^3}{\frac{P_3(y)}{|x-y|^2}\mathrm{d}y}$, we obtain $\Delta(v-\zeta) = 0$. From the asymptotic property of $\zeta(x)$ and $v(x) = o(|x|^2)$ at $\infty$, we can immediately derive from Lemma \ref{lemma Classification of Harmonic Functions} (ii) that, for some constants $a_i \in \mathbb{R}$ ($i = 0, 1, 2, 3$),
$$
v(x) - \zeta(x) = \sum_{i = 1}^3{a_ix_i} + a_0, \qquad \forall x \in \mathbb{R}^3.
$$
Since Lemma \ref{lemma Integral Representation of u for n = 3} implies $|x|^{-2}e^{pv(x)} = |x|^{-2}e^{p\zeta(x)}e^{pa_0}e^{p\sum_{i = 1}^3{a_ix_i}} \in L^1(\mathbb{R}^3)$, we infer from the asymptotic
property of $\zeta(x)$ that $a_1 = a_2 = a_3 = 0$. Hence the integral representation formula for $v$
holds. The asymptotic property of $v$ follows immediately.
\end{proof}

As a consequence of Lemma \ref{lemma Integral Representation of v for n = 3}, we have the following Lemma.
\begin{lem}\label{lemma Asymptotic Property for n = 3}
Assume $n=3$ and $p>0$. Let $(u,v)$ be a pair of classical solutions to the system \eqref{system} such that $u \geqslant 0$, $v(x)=o(|x|^2)$ as $|x|\rightarrow+\infty$ and $u$ satisfies the finite total mass condition \eqref{fm}. Suppose $u(x) = O(|x|^K)$ as $|x|\rightarrow+\infty$ for some $K \gg 1$ arbitrarily large. Then we have, for arbitrarily small $\delta > 0$,
$$
\left\{
\begin{aligned}
\lim_{|x|\to\infty}{\frac{e^{pv(x)}}{|x|^{-\alpha p - \delta}}} & = +\infty, \\
\lim_{|x|\to\infty}{\frac{e^{pv(x)}}{|x|^{-\alpha p + \delta}}} & = 0.
\end{aligned}
\right.
$$
Consequently, $\alpha = \frac{1}{2\pi^2}\int_{\mathbb{R}^3}{P_3(y)\mathrm{d}y} \geqslant \frac{1}{p}$. Furthermore, if $\alpha > \frac{3}{p}$, then
$$
\lim_{|x|\to\infty}{|x|^2u(x)} =\beta,
$$
where
$$
\beta= \frac{1}{2\pi^2}\int_{\mathbb{R}^3}{e^{pv(x)}\mathrm{d}x}.
$$
\end{lem}
\begin{proof}
The asymptotic property of $v$ implies that
$$
v(x) = -\alpha\ln{|x|} + o(\ln{|x|}) \quad\text{as}\; |x| \to \infty.
$$
Therefore, we obtain
$$
e^{pv(x)} = |x|^{-\alpha p}e^{o(\ln{|x|})} \quad\text{as}\; |x| \to \infty.
$$
It follows that, for arbitrarily small $\delta > 0$,
$$
\left\{
\begin{aligned}
\lim_{|x|\to\infty}{\frac{e^{pv(x)}}{|x|^{-\alpha p - \delta}}} & = +\infty, \\
\lim_{|x|\to\infty}{\frac{e^{pv(x)}}{|x|^{-\alpha p + \delta}}} & = 0.
\end{aligned}
\right.
$$
From Lemma \ref{lemma Integral Representation of u for n = 3}, one can easily infer that $\alpha \geqslant \frac{1}{p}$. Furthermore, if we assume $\alpha > \frac{3}{p}$, from the asymptotic property of $v$, it follows immediately that $\frac{1}{2\pi^2}\int_{\mathbb{R}^3}{e^{pv(x)}\mathrm{d}x} < +\infty$.
Next, we prove the asymptotic property of $u$. Let $\delta = \frac{\alpha p-3}{2}$, then there exists a $R_0 \geqslant 1$ sufficiently large such that
$$
e^{pv(x)} \leqslant |x|^{-\frac{\alpha p+3}{2}}, \quad \forall |x| \geqslant R_0.
$$
From the integral representation formula of $u$, we only need to show
$$
\lim_{|x|\to\infty}{\int_{\mathbb{R}^3}{\frac{|x|^2-|x-y|^2}{|x-y|^2}e^{pv(y)}}\mathrm{d}y} = 0.
$$
Indeed, for any $|x| \geqslant R_0$, we have
$$
\begin{aligned}
& \left|\int_{\mathbb{R}^3}{\frac{|x|^2-|x-y|^2}{|x-y|^2}e^{pv(y)}}dy\right| \\
\leqslant & 3\int_{|y-x|<\frac{|x|}{2}}{\frac{1}{|x-y|^2|y|^\frac{\alpha p-1}{2}}\mathrm{d}y} + 3\int_{\{|y-x|\geqslant\frac{|x|}{2}\}\cap\{|y|\geqslant\frac{|x|}{2}\}}{e^{pv(y)}\mathrm{d}y} \\
& \; + \frac{6}{|x|}\int_{\{|y-x|\geqslant\frac{|x|}{2}\}\cap\{|y|<\frac{|x|}{2}\}\cap\{|y|<R_0\}}{|y|e^{pv(y)}\mathrm{d}y} \\
& \; + \frac{6}{|x|}\int_{\{|y-x|\geqslant\frac{|x|}{2}\}\cap \{R_0\leqslant|y|<\frac{|x|}{2}\}}{\frac{1}{|y|^\frac{\alpha p+1}{2}}\mathrm{d}y} \\
\leqslant & \frac{3\times2^\frac{\alpha p+1}{2}\pi}{|x|^\frac{\alpha p-3}{2}} + o_{|x|}(1) + \frac{6}{|x|}\int_{|y|<R_0}{|y|e^{pv(y)}\mathrm{d}y} + \frac{24\pi}{|x|}\xi(x) \\
= & o_{|x|}(1),
\end{aligned}
$$
where
$$
\xi(x)
=
\left\{
\begin{aligned}
& \frac{2}{5-\alpha p}\left(\frac{|x|}{2}\right)^\frac{5-\alpha p}{2}, & \;\text{if}\; & 3 < \alpha p < 5, \\
& \ln{\left(\frac{|x|}{2}\right)}, & \;\text{if}\; & \alpha p = 5, \\
& \frac{2}{\alpha p-5}R_0^\frac{5-\alpha p}{2}, & \;\text{if}\; & \alpha p > 5.
\end{aligned}
\right.
$$
Thus we obtain the asymptotic property for $u$.
\end{proof}

As a consequence, we can prove the following asymptotic property of
$$
\int_{\mathbb{R}^3}{\frac{u^2(y)}{|x-y|^2}\mathrm{d}y},
$$
as $|x|$ tends to $\infty$.

\begin{cor}\label{corollary Asymptotic Property of P for n = 3}
Assume $n=3$ and $p>0$. Let $(u,v)$ be a pair of classical solutions to the system \eqref{system} such that $u \geqslant 0$, $v(x)=o(|x|^2)$ as $|x|\rightarrow+\infty$ and $u$ satisfies the finite total mass condition \eqref{fm}. Suppose $u(x) = O(|x|^K)$ as $|x|\rightarrow+\infty$ for some $K \gg 1$ arbitrarily large. If $\alpha > \frac{3}{p}$, then we have
$$
\lim_{|x|\to\infty}{\int_{\mathbb{R}^3}{\frac{|x|^2u^2(y)}{|x-y|^2}\mathrm{d}y}} = \int_{\mathbb{R}^3}{u^2(y)\mathrm{d}y}.
$$
\end{cor}
\begin{proof}
From the asymptotic property for $u$, there exists $R_0 \geqslant 1$ sufficiently large, for some constant $C \in \mathbb{R}$ such that
$$
u^2(x) \leqslant C|x|^{-4}, \quad \forall |x| \geqslant R_0.
$$
We only need to show
$$
\lim_{|x|\to\infty}{\int_{\mathbb{R}^3}{\frac{|x|^2-|x-y|^2}{|x-y|^2}u^2(y)\mathrm{d}y}} = 0.
$$
Indeed, for $|x|$ large enough, we have
$$
\begin{aligned}
& \left|\int_{\mathbb{R}^3}{\frac{|x|^2-|x-y|^2}{|x-y|^2}u^2(y)\mathrm{d}y}\right| \\
\leqslant & 3\int_{|y-x|<\frac{|x|}{2}}{\frac{C}{|x-y|^2|y|^2}\mathrm{d}y} + 3\int_{\{|y-x|\geqslant\frac{|x|}{2}\}\cap\{|y|\geqslant\frac{|x|}{2}\}}{u^2(y)\mathrm{d}y} \\
& \; + \frac{6}{|x|}\int_{\{|y-x|\geqslant\frac{|x|}{2}\}\cap\{|y|\leqslant\frac{|x|}{2}\}\cap\{|y|<R_0\}}{|y|u^2(y)\mathrm{d}y} \\
& \; + \frac{6}{|x|}\int_{\{|y-x|\geqslant\frac{|x|}{2}\} \cap \{R_0\leqslant|y|\leqslant\frac{|x|}{2}\}}{\frac{C}{|y|^3}\mathrm{d}y} \\
\leqslant & \frac{24C\pi}{|x|} + o_{|x|}(1) + \frac{6}{|x|}\int_{|y|<R_0}{|y|u^2(y)\mathrm{d}y} + \frac{24C\pi}{|x|}\ln\left(\frac{|x|}{2}\right) \\
= & o_{|x|}(1).
\end{aligned}
$$
This finishes our proof of Corollary \ref{corollary Asymptotic Property of P for n = 3}.
\end{proof}

From Lemmas \ref{lemma Integral Representation of u for n = 3} and \ref{lemma Integral Representation of v for n = 3}, we have proved that the classical solution $(u, v)$ of the system \eqref{system} solve the following integral system:
$$
\left\{
\begin{aligned}
& u(x) = \frac{1}{2\pi^2}\int_{\mathbb{R}^3}{\frac{e^{pv(y)}}{|x-y|^2}\mathrm{d}y}, \\
& v(x) = \frac{1}{2\pi^2}\int_{\mathbb{R}^3}{P_3(y)\ln{\left(\frac{|y|}{|x-y|}\right)}\mathrm{d}y} + \gamma,
\end{aligned}
\right.
$$
where $\gamma \in \mathbb{R}$. This completes our proof of Theorem \ref{theorem Integral Representation for n = 3}.

\subsubsection{Proof of Theorem \ref{theorem Integral Representation for n = 4}}

In this subsection, we will prove Theorem \ref{theorem Integral Representation for n = 4}.

\begin{lem}[Integral Representation of $u$]\label{lemma Integral Representation of u for n = 4}
Assume $n = 4$. Let $(u,v)$ be a pair of classical solutions to the system \eqref{system} such that $u \geqslant 0$ and satisfies \eqref{fm}. Then we have the integral representation
$$
u(x) = \frac{1}{4\pi^2}\int_{\mathbb{R}^4}{\frac{e^{pv(y)}}{|x-y|^3}\mathrm{d}y}.
$$
In addition, $u>0$ in $\mathbb{R}^4$, $u(x) \geqslant \frac{c}{|x|^3}$ for some constant $c>0$ and $|x|$ large enough, and
$$
\int_{\mathbb{R}^4}{\frac{e^{pv(y)}}{|y|^3}\mathrm{d}y} < +\infty.
$$
\end{lem}
\begin{proof}
For arbitrary $R>0$, let
$$
u_R(x) \triangleq \int_{B(0,R)}{G_R(x,y)e^{pv(y)}\mathrm{d}y},
$$
where Green's function for $(-\Delta)^\frac{1}{2}$ on $B(0,R) \subset \mathbb{R}^4$ is given by
$$
G_R(x,y) \triangleq \frac{C}{|x-y|^3}\int_{0}^{\frac{t_R}{d_R}}{\frac{1}{(1+z)^2\sqrt{z}}\mathrm{d}z}, \qquad x,y \in B(0,R),
$$
and
$$
\left\{
\begin{aligned}
d_R & = \frac{|x-y|^2}{R^2}, \\
t_R & = \left(1-\frac{|x|^2}{R^2}\right)\left(1-\frac{|y|^2}{R^2}\right), \\
C & = \frac{1}{4\pi^2}\left(\int_{0}^{+\infty}{\frac{1}{(1+z)^2\sqrt{z}}\mathrm{d}z}\right)^{-1}.
\end{aligned}
\right.
$$
We can deduce that $u_R \in C(\mathbb{R}^4) \cap \mathcal{L}_1(\mathbb{R}^4) \cap C^{1,\varepsilon}_{loc}(B(0,R))$ solves
$$
\left\{
\begin{aligned}
& (-\Delta)^\frac{1}{2}u_R = e^{pv} \qquad \text{in} \; B(0,R), \\
& u_R = 0 \qquad\qquad\quad\;\; \text{in} \; \mathbb{R}^4 \backslash B(0,R).
\end{aligned}
\right.
$$
Using Lemma \ref{lemma Maximum Principle}, for any $R > 0$, we have
$$
u \geqslant u_R .
$$
Taking $R\to\infty$, we get
$$
u(x) \geqslant \frac{1}{4\pi^2}\int_{\mathbb{R}^4}{\frac{e^{pv(y)}}{|x-y|^3}\mathrm{d}y} \triangleq \overline{u}(x).
$$
Then by Lemma \ref{lemma Liouville Theorem}, we get
$$
u(x) = \overline{u}(x) + C_u \geqslant 0.
$$
If $u(x) \geqslant C_u > 0$, then
$$
\iint_{\mathbb{R}^4\times\mathbb{R}^4}{\frac{u^2(x)u^2(y)}{|x-y|^2}\mathrm{d}x\mathrm{d}y} > \iint_{\mathbb{R}^4\times\mathbb{R}^4}{\frac{C_u^4}{|x-y|^2}\mathrm{d}x\mathrm{d}y} = +\infty,
$$
which is a contradiction. Thus $C_u$ must be zero, i.e.,
$$
u(x) = \frac{1}{4\pi^2}\int_{\mathbb{R}^4}{\frac{e^{pv(y)}}{|x-y|^3}\mathrm{d}y},
$$
and hence $u>0$ in $\mathbb{R}^4$, and
$$
\int_{\mathbb{R}^4}{\frac{e^{pv(y)}}{|y|^3}\mathrm{d}y}=4\pi^2 u(0)< +\infty.
$$
Moreover, For $x$ large enough,
$$
\begin{aligned}
& u(x) \\
\geqslant & \frac{1}{4\pi^2}\int_{\mathbb{R}^4}{\frac{e^{pv(y)}}{|x-y|^3}\mathrm{d}y} \\
\geqslant & \frac{1}{4\pi^2}\int_{1 \leqslant |y| < \frac{|x|}{2}}{\frac{e^{pv(y)}}{|x-y|^3}\mathrm{d}y} \\
\geqslant & \frac{2}{27\pi^2|x|^3}\int_{1 \leqslant |y| < \frac{|x|}{2}}{\frac{e^{pv(y)}}{|y|^2}\mathrm{d}y} \\
\geqslant & \frac{1}{27\pi^2|x|^3}\int_{|y| \geqslant 1}{\frac{e^{pv(y)}}{|y|^2}\mathrm{d}y} \\
\geqslant & \frac{C}{|x|^3}.
\end{aligned}
$$
\end{proof}

Since we have \eqref{fm}, we can define
$$
\zeta(x) \triangleq \frac{1}{8\pi^2}\int_{\mathbb{R}^4}{P_4(y)\ln{\left(\frac{|y|}{|x-y|}\right)}\mathrm{d}y}.
$$

Next, we need to prove the integral representation formula and the asymptotic property of $v$. Before that, we need the following lemmas.

\begin{lem}\label{lemma Lowerbound of xi(x) for n = 4}
Assume $n = 4$. Let $(u,v)$ be a pair of classical solutions to the system \eqref{system} such that $u \geqslant 0$ and satisfies \eqref{fm}. Then there is a constant $C$ such that when $|x|$ is sufficiently large, one has
$$
-\zeta(x) \leqslant \alpha\ln{|x|}+C.
$$
where $\alpha \triangleq \frac{1}{8\pi^2}\int_{\mathbb{R}^4}{P_4(y)\mathrm{d}y} \in (0,+\infty)$.
\end{lem}
\begin{proof}
Fix $x$ such that $|x|$ large enough, let us denote
$$
\begin{aligned}
& S_1 = \left\{y \left| |x-y| \leqslant \frac{|x|}{2} \right.\right\}, \\
& S_2 = \left\{y \left| |x-y| \geqslant \frac{|x|}{2} \right.\right\}.
\end{aligned}
$$
It is easy to check that $\mathbb{R}^4 = S_1 \cup S_2$ and for $y \in S_1$, one has $|y| \geqslant |x|-|x-y| \geqslant \frac{|x|}{2} \geqslant |x-y|$, which indicates $\ln{\left(\frac{|x-y|}{|y|}\right)} \leqslant 0$. For $y \in S_2$ and $|y| \geqslant 2$, one has $|x-y| \leqslant |x|+|y| \leqslant |x||y|$ (note that $|x|$ is sufficiently large). Finally, for $y \in S_2$ and $|y| \leqslant 2$, one has $|x-y| \leqslant |x|$ (note that $|x|$ is sufficiently large.), which also indicates $\ln{\left(\frac{|x-y|}{|y|}\right)} \leqslant 0$.
Therefore,
$$
\begin{aligned}
-\zeta(x) & = \frac{1}{8\pi^2}\int_{\mathbb{R}^4}{P_4(y)\ln{\left(\frac{|x-y|}{|y|}\right)}\mathrm{d}y} \\
& \leqslant \frac{1}{8\pi^2}\int_{|x-y| \geqslant \frac{|x|}{2}}{P_4(y)\ln{\left(\frac{|x-y|}{|y|}\right)}\mathrm{d}y} \\
& \leqslant \frac{\ln{|x|}}{8\pi^2}\int_{\{|x-y| \geqslant \frac{|x|}{2}\} \cap \{|y| \geqslant 2\}}{P_4(y)\mathrm{d}y} + C \\
& \leqslant \alpha\ln{|x|}+C.
\end{aligned}
$$
\end{proof}

Now we can prove the integral representation for $\Delta{v}$.

\begin{lem}\label{lemma Integral Representation of Delta v for n = 4}
Assume $n = 4$. Let $(u,v)$ be a pair of classical solutions to the system \eqref{system} such that $u \geqslant 0$ and satisfies \eqref{fm}. Then
$$
\Delta{v}(x) = -\frac{1}{4\pi^2}\int_{\mathbb{R}^4}{\frac{P_4(y)}{|x-y|^2}\mathrm{d}y}-C
$$
for some constant $C \geqslant 0$.
\end{lem}
\begin{proof}
Denote $m(x) = v(x)-\zeta(x)$, from \eqref{system}, we have
$$
\Delta^2{m} = 0,
$$
i.e., $\Delta{m}$ is harmonic in $\mathbb{R}^4$. From the mean value
theorem of harmonic function, for any $x_0 \in \mathbb{R}^3$ and $r > 0$, we derive
$$
\begin{aligned}
\Delta{m}(x_0) & = \frac{2}{\pi^2r^4}\int_{B(x_0,r)}{\Delta{m}(y)\mathrm{d}y} \\
& = \frac{2}{\pi^2r^4}\int_{\partial B(x_0,r)}{\frac{\partial m}{\partial r}\mathrm{d}\sigma}.
\end{aligned}
$$
Integrating from $0$ to $r$, we obtain
$$
\frac{r^2}{8}\Delta{m}(x_0) = \fint_{\partial B(x_0,r)}{m(y)\mathrm{d}\sigma} - m(x_0).
$$
Using the Jensen's inequality, we have
$$
\begin{aligned}
e^{\frac{r^2}{8}\Delta{m}(x_0)} & = e^{-m(x_0)}e^{\fint_{\partial B(x_0,r)}{m(y)\mathrm{d}\sigma}} \\
& \leqslant e^{-m(x_0)}\fint_{\partial B(x_0,r)}{e^{m(y)}\mathrm{d}\sigma}.
\end{aligned}
$$
By Lemma \ref{lemma Lowerbound of xi(x) for n = 4}, we have $m(x) = v(x) - \zeta(x) \leqslant v(x) + \alpha\ln{|x|}+C$. Combining this with the above inequality and $|x|^{-3}e^{pv(x)} \in L^1(\mathbb{R}^4)$ in Lemma \ref{lemma Integral Representation of u for n = 4}, we get
$$
r^{-p\alpha}e^{\frac{pr^2}{8}\Delta{m}(x_0)} \in L^1_r([1,+\infty)).
$$
Thus $\Delta{m}(x_0) \leqslant 0$ for all $x_0 \in \mathbb{R}^4$. By Liouville Theorem, we have $\Delta{m}(x) = -C$ in $\mathbb{R}^4$ for some constant $C \geqslant 0$. Together with $\Delta{\zeta}(x) = -\frac{1}{4\pi^2}\int_{\mathbb{R}^4}{\frac{P_4(y)}{|x-y|^2}\mathrm{d}y}$, we get $\Delta{v}(x) = -\frac{1}{4\pi^2}\int_{\mathbb{R}^4}{\frac{P_4(y)}{|x-y|^2}\mathrm{d}y} - C$.
\end{proof}

Now we need to assume $v(x) = o(|x|^2)$ at $\infty$, and we can get the precise integral representation formula for $\Delta{v}(x)$.

\begin{lem}\label{lemma Precise Integral Representation of Delta v for n = 4}
Assume $n = 4$. Let $(u,v)$ be a pair of classical solutions to the system \eqref{system} such that $u \geqslant 0$ satisfying \eqref{fm}, and $v(x) = o(|x|^2)$ at $\infty$. Then we have
$$
\Delta{v}(x) = -\frac{1}{4\pi^2}\int_{\mathbb{R}^4}{\frac{P_4(y)}{|x-y|^2}\mathrm{d}y}.
$$
\end{lem}
\begin{proof}
By Lemma \ref{lemma Integral Representation of Delta v for n = 4}, we have
$$
\Delta{v}(x) = -\frac{1}{4\pi^2}\int_{\mathbb{R}^4}{\frac{P_4(y)}{|x-y|^2}\mathrm{d}y} - C.
$$
If $C > 0$, take $\varepsilon \in (0,\frac{C}{16})$ and $R_0$ large enough such that
$$
\begin{aligned}
& \Delta{v}(x) \leqslant -C < 0 \;\text{in}\; \mathbb{R}^4, \\
& v(y) + \varepsilon|y|^2 \geqslant \frac{\varepsilon}{2}|y|^2 \geqslant 0 \quad \text{for}\; |y| \geqslant R_0.
\end{aligned}
$$
For $B > 0$, define
$$
n(y) = v(y) + \varepsilon|y|^2+B(|y|^{-2}-R_0^{-2}).
$$
Thus
$$
\Delta{n}(y) = \Delta{v}(y) + 8\varepsilon < -\frac{C}{2} < 0
$$
for $|y| \geqslant R_0$. Furthermore, we get
$$
\begin{aligned}
& \lim_{|y|\to+\infty}{n(y)} = +\infty \;\text{for any}\; B > 0, \\
& \lim_{B\to+\infty}{n(y)} = -\infty \;\text{for any}\; y \in \mathbb{R}^4 \backslash \overline{B(0,R_0)}.
\end{aligned}
$$
Thus, we can take $B$ sufficiently large such that
$\displaystyle \inf_{|y| \geqslant R_0}{n(y)}$ is attained by some $y_0 \in \mathbb{R}^4$ with $|y_0| > R_0$. By applying the maximum principle to $n(y)$, we derive a contradiction. This implies $C = 0$.
\end{proof}

Finally, we can prove the integral representation formula and the asymptotic property for $v$.
\begin{lem}[Integral Representation of $v$]\label{lemma Integral Representation of v for n = 4}
Assume $n=4$ and $p>0$. Let $(u,v)$ be a pair of classical solutions to the system \eqref{system} such that $u \geqslant 0$, $v(x)=o(|x|^2)$ as $|x|\rightarrow+\infty$ and $u$ satisfies the finite total mass condition \eqref{fm}. Suppose $u(x) = O(|x|^K)$ as $|x|\rightarrow+\infty$ for some $K \gg 1$ arbitrarily large. Then we have
$$
v(x) = \zeta(x)+\gamma = \frac{1}{8\pi^2}\int_{\mathbb{R}^4}{P_4(y)\ln{\left(\frac{|y|}{|x-y|}\right)}\mathrm{d}y}+\gamma,
$$
where $\gamma \in \mathbb{R}$ is a constant. Moreover,
$$
\lim_{|x|\to\infty}{\frac{v(x)}{\ln{|x|}}} = -\alpha,
$$
where $\alpha \triangleq \frac{1}{8\pi^2}\int_{\mathbb{R}^4}{P_4(y)\mathrm{d}y} \in (0,+\infty)$.
\end{lem}
\begin{proof}
We will first prove the following asymptotic property:
$$
\lim_{|x|\to\infty}{\frac{\zeta(x)}{\ln{|x|}}} = -\alpha.
$$
To this end, we only need to show that
$$
\lim_{|x|\to\infty}{\int_{\mathbb{R}^4}{P_4(y)\frac{\ln{|x-y|}-\ln{|y|}-\ln{|x|}}{\ln{|x|}}\mathrm{d}y}} = 0.
$$
By using the $\exp^L+L\ln{L}$ inequality, we get
$$
\begin{aligned}
& \int_{B(x,1)}{P_4(y)\ln\left(\frac{1}{|x-y|}\right)\mathrm{d}y} \\
\leqslant & \int_{B(x,1)}{\frac{1}{|x-y|}\mathrm{d}y}+\int_{B(x,1)}{P_4(y)\ln\left(P_4(y)+1\right)\mathrm{d}y} \\
\leqslant & \frac{2\pi^2}{3}+\left[\max_{|y-x|\leqslant1}\ln\left(P_4(y)+1\right)\right]\int_{B(x,1)}{P_4(y)\mathrm{d}y}.
\end{aligned}
$$
As a consequence, by above inequality, the conditions \eqref{fm} and $u = O(|x|^K) $ at $\infty$ for some $K \gg 1$ arbitrarily large, we have, for any $|x| \geqslant e^2$ large enough,
$$
\begin{aligned}
& \left|\int_{\mathbb{R}^4}{P_4(y)\frac{\ln{|x-y|}-\ln{|y|}-\ln{|x|}}{\ln{|x|}}\mathrm{d}y}\right| \\
\leqslant & 3\int_{B(x,1)}{P_4(y)\mathrm{d}y}+\frac{2\pi^2}{3\ln{|x|}}+\frac{O(2K\ln{|x|})}{\ln{|x|}}\int_{B(x,1)}{P_4(y)\mathrm{d}y} \\
& \; + \frac{\displaystyle \max_{|y|\leqslant\ln{|x|}}{\left|\ln{\frac{|x-y|}{|x|}}\right|}}{\ln{|x|}}\int_{|y|<\ln{|x|}}{P_4(y)\mathrm{d}y}+\frac{1}{\ln{|x|}}\int_{|y|<\ln{|x|}}{|\ln{|y|}|P_4(y)\mathrm{d}y} \\
& \; + \sup_{\substack{|y-x|\geqslant1 \\ |y|\geqslant\ln{|x|}}}{\frac{|\ln{|x-y|}-\ln{|y|}-\ln{|x|}|}{\ln{|x|}}}\int_{|y|\geqslant\ln{|x|}}{P_4(y)\mathrm{d}y} \\
\leqslant & o_{|x|}(1)+\frac{2\pi^2}{3\ln{|x|}}+\frac{\ln{2}}{\ln{|x|}}\int_{\mathbb{R}^4}{P_4(y)\mathrm{d}y}+\frac{1}{\ln{|x|}}\int_{|y|<1}{\ln\left(\frac{1}{|y|}\right)P_4(y)\mathrm{d}y} \\
& \; + \frac{\ln(\ln{|x|})}{\ln{|x|}}\int_{\mathbb{R}^4}{P_4(y)\mathrm{d}y}+\left(2+\frac{\ln{2}}{\ln{|x|}}\right)\int_{|y|\geqslant\ln{|x|}}{P_4(y)\mathrm{d}y} \\
= & o_{|x|}(1),
\end{aligned}
$$
where we have used the facts $1 > \frac{1}{|x|}+\frac{1}{|y|} \geqslant \frac{|x-y|}{|x|\cdot|y|} \geqslant \frac{1}{2|x|^2}$ for any $|y-x| \geqslant 1$ and $|y| \geqslant \ln{|x|}$. By letting $|x| \to +\infty$, we obtain
$$
\lim_{|x|\to\infty}{\int_{\mathbb{R}^4}{P_4(y)\frac{\ln{|x-y|}-\ln{|y|}-\ln{|x|}}{\ln{|x|}}\mathrm{d}y}} = 0.
$$
By Lemma \ref{lemma Precise Integral Representation of Delta v for n = 3} and $\Delta{\zeta}(x) = -\frac{1}{8\pi^2}\int_{\mathbb{R}^4}{\frac{P_4(y)}{|x-y|^3}\mathrm{d}y}$, we obtain $\Delta(v-\zeta) = 0$. From the asymptotic property of $\zeta(x)$ and $v(x) = o(|x|^2)$ at $\infty$, we can immediately derive from Lemma \ref{lemma Classification of Harmonic Functions} (ii) that, for some constants $a_i \in \mathbb{R}$ ($i = 0, 1, 2, 3, 4$),
$$
v(x) - \zeta(x) = \sum_{i = 1}^4{a_ix_i} + a_0, \qquad \forall x \in \mathbb{R}^4.
$$
Since Lemma \ref{lemma Integral Representation of u for n = 3} implies $|x|^{-3}e^{pv(x)} = |x|^{-2}e^{p\zeta(x)}e^{pa_0}e^{p\sum_{i = 1}^4{a_ix_i}} \in L^1(\mathbb{R}^4)$, we infer from the asymptotic
property of $\zeta(x)$ that $a_1 = a_2 = a_3 = a_4 = 0$. Hence the integral representation formula for $v$
holds. The asymptotic property of $v$ follows immediately.
\end{proof}

As a consequence of Lemma \ref{lemma Integral Representation of v for n = 4}, we have the following Lemma.

\begin{lem}\label{lemma Asymptotic Property for n = 4}
Assume $n=4$ and $p>0$. Let $(u,v)$ be a pair of classical solutions to the system \eqref{system} such that $u \geqslant 0$, $v(x)=o(|x|^2)$ as $|x|\rightarrow+\infty$ and $u$ satisfies the finite total mass condition \eqref{fm}. Suppose $u(x) = O(|x|^K)$ as $|x|\rightarrow+\infty$ for some $K \gg 1$ arbitrarily large. Then we have, for arbitrarily small $\delta > 0$,
$$
\left\{
\begin{aligned}
\lim_{|x|\to\infty}{\frac{e^{pv(x)}}{|x|^{-\alpha p - \delta}}} & = +\infty, \\
\lim_{|x|\to\infty}{\frac{e^{pv(x)}}{|x|^{-\alpha p + \delta}}} & = 0.
\end{aligned}
\right.
$$
Consequently, $\alpha = \frac{1}{8\pi^2}\int_{\mathbb{R}^4}{P_4(y)\mathrm{d}y} \geqslant \frac{1}{p}$. Furthermore, if $\alpha > \frac{4}{p}$, then
$$
\lim_{|x|\to\infty}{|x|^3u(x)} = \beta,
$$
where
$$
\beta= \frac{1}{4\pi^2}\int_{\mathbb{R}^4}{e^{pv(x)}\mathrm{d}x}.
$$
\end{lem}
\begin{proof}
The asymptotic property of $v$ implies that
$$
v(x) = -\alpha\ln{|x|} + o(\ln{|x|}) \quad\text{as}\; |x| \to \infty.
$$
Therefore, we obtain
$$
e^{pv(x)} = |x|^{-\alpha p}e^{o(\ln{|x|})} \quad\text{as}\; |x| \to \infty.
$$
Therefore, for arbitrarily small $\delta > 0$,
$$
\left\{
\begin{aligned}
\lim_{|x|\to\infty}{\frac{e^{pv(x)}}{|x|^{-\alpha p - \delta}}} & = +\infty \\
\lim_{|x|\to\infty}{\frac{e^{pv(x)}}{|x|^{-\alpha p + \delta}}} & = 0.
\end{aligned}
\right.
$$
From Lemma \ref{lemma Integral Representation of u for n = 4}, one can easily infer that $\alpha \geqslant \frac{1}{p}$. Furthermore, if we assume $\alpha > \frac{4}{p}$, from the asymptotic property of $v$, it follows immediately that $\frac{1}{4\pi^2}\int_{\mathbb{R}^4}{e^{pv(x)}\mathrm{d}x} < +\infty$.
Next, we prove the asymptotic property of $u$. Let $\delta = \frac{\alpha p-4}{2}$, then there exists a $R_0 \geqslant 1$ sufficiently large such that
$$
e^{pv(x)} \leqslant |x|^{-\frac{\alpha p+4}{2}}, \quad \forall |x| \geqslant R_0.
$$
From the integral representation formula of $u$, we only need to show
$$
\lim_{|x|\to\infty}{\int_{\mathbb{R}^4}{\frac{|x|^3-|x-y|^3}{|x-y|^3}e^{pv(y)}}\mathrm{d}y} = 0.
$$
Indeed, for any $|x| \geqslant R_0$, we have
$$
\begin{aligned}
& \left|\int_{\mathbb{R}^4}{\frac{|x|^3-|x-y|^3}{|x-y|^3}e^{pv(y)}}dy\right| \\
\leqslant & 7\int_{|y-x|<\frac{|x|}{2}}{\frac{1}{|x-y|^3|y|^\frac{\alpha p-2}{2}}\mathrm{d}y} + 7\int_{|y-x|\geqslant\frac{|x|}{2}\wedge|y|\geqslant\frac{|x|}{2}}{e^{pv(y)}\mathrm{d}y} \\
& \; + \frac{14}{|x|}\int_{|y-x|\geqslant\frac{|x|}{2}\wedge|y|<\frac{|x|}{2}\wedge|y|<R_0}{|y|e^{pv(y)}\mathrm{d}y} \\
& \; + \frac{14}{|x|}\int_{|y-x|\geqslant\frac{|x|}{2}\wedge R_0\leqslant|y|<\frac{|x|}{2}}{\frac{1}{|y|^\frac{\alpha p+2}{2}}\mathrm{d}y} \\
\leqslant & \frac{7\times2^\frac{\alpha p-2}{2}\pi^2}{|x|^\frac{\alpha p-4}{2}} + o_{|x|}(1) + \frac{14}{|x|}\int_{|y|<R_0}{|y|e^{pv(y)}\mathrm{d}y} + \frac{28\pi^2}{|x|}\xi(x) \\
 = & o_{|x|}(1),
\end{aligned}
$$
where
$$
\xi(x)
 =
\left\{
\begin{aligned}
& \frac{2}{6-\alpha p}\left(\frac{|x|}{2}\right)^\frac{6-\alpha p}{2}, & \;\text{if}\; & 4 < \alpha p < 6, \\
& \ln{\left(\frac{|x|}{2}\right)}, & \;\text{if}\; & \alpha p = 6, \\
& \frac{2}{\alpha p-6}R_0^\frac{6-\alpha p}{2}, & \;\text{if}\; & \alpha p > 6.
\end{aligned}
\right.
$$
Thus we obtained the asymptotic property for $u$.
\end{proof}

As a consequence, we can prove the following asymptotic property of
$$
\int_{\mathbb{R}^4}{\frac{u^2(y)}{|x-y|^2}\mathrm{d}y},
$$
as $|x|$ tends to $\infty$.

\begin{cor}\label{corollary Asymptotic Property of P for n = 4}
Assume $n=4$ and $p>0$. Let $(u,v)$ be a pair of classical solutions to the system \eqref{system} such that $u \geqslant 0$, $v(x)=o(|x|^2)$ as $|x|\rightarrow+\infty$ and $u$ satisfies the finite total mass condition \eqref{fm}. Suppose $u(x) = O(|x|^K)$ as $|x|\rightarrow+\infty$ for some $K \gg 1$ arbitrarily large. If $\alpha > \frac{4}{p}$, then we have
$$
\lim_{|x|\to\infty}{\int_{\mathbb{R}^4}{\frac{|x|^2u^2(y)}{|x-y|^2}\mathrm{d}y}} = \int_{\mathbb{R}^4}{u^2(y)\mathrm{d}y}.
$$
\end{cor}
\begin{proof}
From the asymptotic property for $u$, there exists $R_0 \geqslant 1$ sufficiently large, for some constant $C \in \mathbb{R}$ such that
$$
u^2(x) \leqslant C|x|^{-6}, \quad \forall |x| \geqslant R_0.
$$
We only need to show
$$
\lim_{|x|\to\infty}{\int_{\mathbb{R}^4}{\frac{|x|^2-|x-y|^2}{|x-y|^2}u^2(y)\mathrm{d}y}} = 0.
$$
Indeed, for $|x|$ large enough, we have
$$
\begin{aligned}
& \left|\int_{\mathbb{R}^4}{\frac{|x|^2-|x-y|^2}{|x-y|^2}u^2(y)\mathrm{d}y}\right| \\
\leqslant & 3\int_{|y-x|<\frac{|x|}{2}}{\frac{C}{|x-y|^2|y|^4}\mathrm{d}y} + 3\int_{|y-x|\geqslant\frac{|x|}{2}\wedge|y|\geqslant\frac{|x|}{2}}{u^2(y)\mathrm{d}y} \\
& \; + \frac{6}{|x|}\int_{|y-x|\geqslant\frac{|x|}{2}\wedge|y|\leqslant\frac{|x|}{2}\wedge|y|<R_0}{|y|u^2(y)\mathrm{d}y} \\
& \; + \frac{6}{|x|}\int_{|y-x|\geqslant\frac{|x|}{2} \wedge R_0\leqslant|y|\leqslant\frac{|x|}{2}}{\frac{C}{|y|^5}\mathrm{d}y} \\
\leqslant & \frac{12C\pi^2}{|x|^2} + o_{|x|}(1) + \frac{6}{|x|}\int_{|y|<R_0}{|y|u^2(y)\mathrm{d}y} + \frac{12C\pi^2}{R_0|x|} \\
= & o_{|x|}(1).
\end{aligned}
$$
This finishes our proof of Corollary \ref{corollary Asymptotic Property of P for n = 4}.
\end{proof}

From Lemmas \ref{lemma Integral Representation of u for n = 4} and \ref{lemma Integral Representation of v for n = 4}, we have proved that the classical solution $(u, v)$ of the system \eqref{system} solve the following integral system:
$$
\left\{
\begin{aligned}
& u(x) = \frac{1}{4\pi^2}\int_{\mathbb{R}^4}{\frac{e^{pv(y)}}{|x-y|^3}\mathrm{d}y}, \\
& v(x) = \frac{1}{8\pi^2}\int_{\mathbb{R}^4}{P_4(y)\ln{\left(\frac{|y|}{|x-y|}\right)}\mathrm{d}y} + \gamma,
\end{aligned}
\right.
$$
where $\gamma \in \mathbb{R}$. This concludes our proof of Theorem \ref{theorem Integral Representation for n = 3}.

\subsection{Completion of the proof of Theorem \ref{theorem Main Result}}

Assume $(u,v)$ is a pair of classical solutions to the $3,4$-D system \eqref{system} with $u \geq 0$. One should observe that if $(u, v)$ solve the system \eqref{system} with $n=3,4$ for any given $p \in (0,+\infty )$, then $\widetilde{u}:=p^{\frac{1}{4}}u$ and $\widetilde{v}:=pv+\frac{1}{4}\ln p$ solve \eqref{system} with $p = 1$. Due to this observation, we will take $p = 1$ in \eqref{system} hereafter in subsection 3.2 in order to make the notation lighter.

First, from Lemmas \ref{theorem Integral Representation for n = 3} and \ref{theorem Integral Representation for n = 4}, we have
$$
\left\{
\begin{aligned}
& u(\xi) = C_{u,n}\int_{\mathbb{R}^n}{\frac{e^{v(y)}}{|\xi-y|^{n-1}}\mathrm{d}y}, \\
& v(\xi) = C_{v,n}\int_{\mathbb{R}^n}{\ln\left(\frac{|y|}{|\xi-y|}\right)P(y)u^2(y)\mathrm{d}y}+\gamma,
\end{aligned}
\right.
$$
where $n=3,4$ and
$$
\begin{aligned}
& C_{u,3} = \frac{1}{2\pi^2}, \qquad C_{v,3} = \frac{1}{2\pi^2}, \\
& C_{u,4} = \frac{1}{4\pi^2}, \qquad C_{v,4} = \frac{1}{8\pi^2}, \\
& P(\xi) = \int_{\mathbb{R}^n}{\frac{u^2(z)}{|\xi-z|^2}\mathrm{d}z}.
\end{aligned}
$$
We define
$$
\xi^{x,\lambda} = x+\frac{\lambda^2(\xi-x)}{|\xi-x|^2},
$$
\begin{equation}\label{Kelvin Transformation}
\left\{
\begin{aligned}
& u_{x,\lambda}(\xi) = \left(\frac{\lambda}{|\xi-x|}\right)^{n-1}u(\xi^{x,\lambda}), \\
& v_{x,\lambda}(\xi) = (n+1)\ln\left(\frac{\lambda}{|\xi-x|}\right)+v(\xi^{x,\lambda}), \\
& P_{x,\lambda}(\xi) = \left(\frac{\lambda}{|\xi-x|}\right)^2P(\xi^{x,\lambda}),
\end{aligned}
\right.
\end{equation}
as well as
$$
\left\{
\begin{aligned}
& K(x,\lambda;\xi,z) = \frac{1}{|\xi-z|^{n-1}}- \left(\frac{\lambda}{|\xi-x|}\right)^{n-1}\frac{1}{|\xi^{x,\lambda}-z|^{n-1}}, \\
& M(x,\lambda;\xi,z) = \frac{1}{|\xi-z|^2}- \left(\frac{\lambda}{|\xi-x|}\right)^2\frac{1}{|\xi^{x,\lambda}-z|^2}, \\
& L(x,\lambda;\xi,z) = \ln\left(\frac{|\xi^{x,\lambda}-z|}{|\xi-z|}\frac{|\xi-x|}{\lambda}\right) = \frac{1}{2}\ln\left(1+\frac{\left(\lambda-\frac{|\xi-x|^2}{\lambda}\right)\left(\lambda-\frac{|z-x|^2}{\lambda}\right)}{|\xi-z|^2}\right).
\end{aligned}
\right.
$$
One can verify that $K,M,L>0$ for any $\xi,z\in B(x,\lambda) \backslash \{x\}$ and that
$$
\left\{
\begin{aligned}
u(\xi)-u_{x,\lambda}(\xi) &= C_{u,n}\int_{B(x,\lambda)}{K(x,\lambda;\xi,z)\left(e^{v(z)}-e^{v_{x,\lambda}(z)}\right)\mathrm{d}z}, \\
v(\xi)-v_{x,\lambda}(\xi) &= \left[\alpha-n-1\right]\ln\left(\frac{\lambda}{|\xi-x|}\right) \\
& \quad + C_{v,n}\int_{B(x,\lambda)}{L(x,\lambda;\xi,z)\left(P(z)u^2(z)-P_{x,\lambda}(z)u_{x,\lambda}^2(z)\right)\mathrm{d}z}, \\
P(\xi)-P_{x,\lambda}(\xi) &= \int_{B(x,\lambda)}{M(x,\lambda;\xi,z)(u^2(z)-u_{x,\lambda}^2(z))\mathrm{d}z},
\end{aligned}
\right.
$$
where $\alpha = C_{v,n}\int_{\mathbb{R}^n}P(y)u^2(y)dy < \infty$.

\medskip

In what follows, the two different cases $\alpha\geq n+1$ and $1\leq \alpha\leq n+1$ will be discussed separately.

\medskip

\emph{Step 1}. Start moving the circle $S_{\lambda}(x):=\partial B(x,\lambda)$ from near $\lambda=0$ or $\lambda=+\infty$.

\textit{Case} (i) $1\leq \alpha\leq n+1$. We will show that, for any $x \in \mathbb{R}^n$, there exists $\lambda_0(x) > 0$ small enough, such that for $\lambda < \lambda_0$ and for any $\xi \in B(x,\lambda) \backslash \{x\}$,
$$
u(\xi)-u_{x,\lambda}(\xi) \leqslant 0\quad\text{and}\quad v(\xi)-v_{x,\lambda}(\xi) \leqslant 0.
$$

Let
$$
\begin{aligned}
B_u^+(x,\lambda) & = \{\xi \in B(x,\lambda) \backslash \{x\} | u(\xi)-u_{x,\lambda}(\xi) > 0\}, \\
B_v^+(x,\lambda) & = \{\xi \in B(x,\lambda) \backslash \{x\} | v(\xi)-v_{x,\lambda}(\xi) > 0\}, \\
E(x,\lambda) & = \{\xi \in B(x,\lambda) \backslash \{x\} | (P(\xi)u^2(\xi)-P_{x,\lambda}(\xi)u_{x,\lambda}^2(\xi) > 0\}, \\
E_u^+(x,\lambda) & = E(x,\lambda) \cap B_u^+(x,\lambda), \\
E_P^+(x,\lambda) & = E(x,\lambda) \cap \{\xi \in B(x,\lambda) \backslash \{x\} | P(\xi)-P_{x,\lambda}(\xi) > 0\}.
\end{aligned}
$$
For $\xi \in B_u^+(x,\lambda)$, we have that
$$
\begin{aligned}
0 & < u(\xi)-u_{x,\lambda}(\xi) \\
& \leqslant C_{u,n}\int_{B_v^+(x,\lambda)}{K(x,\lambda;\xi,z)\left(e^{v(z)}-e^{v_{x,\lambda}(z)}\right) \mathrm{d}z} \\
& < C_{u,n}\int_{B_v^+(x,\lambda)}{K(x,\lambda;\xi,z)e^{\eta(z)}(v(z)-v_{x,\lambda}(z))\mathrm{d}z} \\
& < C_{u,n}\int_{B_v^+(x,\lambda)}{\frac{1}{|\xi-z|^{n-1}}e^{\eta(z)}(v(z)-v_{x,\lambda}(z))\mathrm{d}z},
\end{aligned}
$$
where $v_{x,\lambda}(z)<\eta(z)< v(z)$. So according to Hardy-Littlewood-Sobolev Inequality in Lemma \ref{lemma Hardy-Littlewood-Sobolev Inequality}, for some $q \in\left(\frac{n}{n-1},\frac{n}{n-2}\right)$, there exists $K > 0$, such that
$$
\begin{aligned}
\lVert{u-u_{x,\lambda}}\rVert_{L^{q}(B_u^+(x,\lambda))}&\leqslant K\lVert{(v-v_{x,\lambda})e^{\eta(\cdot)}}\rVert_{L^p(B_v^+(x,\lambda))} \\
& \leqslant K\lVert{e^{\eta(\cdot)}}\rVert_{L^q(B_v^+(x,\lambda))} \lVert{v-v_{x,\lambda}}\rVert_{L^n(B_v^+(x,\lambda))},
\end{aligned}
$$
where $\frac{1}{p}-\frac{1}{q} = \frac{1}{n}$. Since $e^{\eta(z)}<e^{v(z)}$ and $v(z)$ is at least $C^1$ near $x$, we may choose $\lambda_1$ small enough, such that for $\lambda < \lambda_1$,
$\lVert{e^{\eta(\cdot)}}\rVert_{L^q(B_v^+(x,\lambda))}\leqslant \frac{1}{10K}$. That is,
$$
\lVert{u-u_{x,\lambda}}\rVert_{L^{q}(B_u^+(x,\lambda))} \leqslant \frac{1}{10}\lVert{v-v_{x,\lambda}}\rVert_{L^n(B_v^+(x,\lambda))}.
$$

Note that, for arbitrary $\varepsilon>0$,
$$
\ln(1+t) = o(t^\varepsilon)\quad \;\text{as}\; t\to\infty.
$$
So for any given $\varepsilon$, there exists $\delta = \delta(\varepsilon)>0$ such that
$$
\ln(1+t)\leqslant t^\varepsilon , \quad \forall t>\frac{1}{\delta^2}.
$$
Thus we can choose $C = C(\varepsilon)>0$ such that
$$
\ln(1+t)\leqslant 2Ct^\varepsilon , \quad \forall t>\frac{1}{4}.
$$
For $\xi,z\in B(x,\lambda)\backslash\{x\}$, one has $|\xi-z|\leqslant 2\lambda$, so
$$
0<L(x,\lambda;\xi,z)\leqslant \frac{1}{2}\ln\left(1+\frac{\lambda^2}{|\xi-z|^2}\right) \leqslant C\frac{\lambda^{2\varepsilon}}{|\xi-z|^{2\varepsilon}}.
$$
Thus for $\xi \in B_v^+(x,\lambda)$, we have
$$
\begin{aligned}
0 & < v(\xi)-v_{x,\lambda}(\xi)  < C_{v,n}\int_{E(x,\lambda)}{L(x,\lambda;\xi,z)\left(P(z)u^2(z)-P_{x,\lambda}(z)u_{x,\lambda}^2(z)\right)\mathrm{d}z} \\
& < C_{v,n}\int_{E_P^+(x,\lambda)}{L(x,\lambda;\xi,z)u^2(z)[P(z)-P_{x,\lambda}(z)]\mathrm{d}z} \\
& \qquad + C_{v,n}\int_{E_u^+(x,\lambda)}{L(x,\lambda;\xi,z)P_{x,\lambda}(z)[u^2(z)-u_{x,\lambda}^2(z)]\mathrm{d}z} \\
& < C_{v,n}C(\varepsilon)\int_{E_P^+(x,\lambda)}{\frac{\lambda^{2\varepsilon}}{|\xi-z|^{2\varepsilon}}u^2(z)[P(z)-P_{x,\lambda}(z)]\mathrm{d}z} \\
& \qquad + C_{v,n}C(\varepsilon)\int_{E_u^+(x,\lambda)}{\frac{\lambda^{2\varepsilon}}{|\xi-z|^{2\varepsilon}}P_{x,\lambda}(z)[u^2(z)-u_{x,\lambda}^2(z)]\mathrm{d}z}.
\end{aligned}
$$

Because $P_{x,\lambda} < P$ on $E_P^+(x,\lambda)$, we have
$$
\begin{aligned}
0 & < P(\xi)-P_{x,\lambda}(\xi) \leqslant \int_{B_u^+(x,\lambda)}{M(x,\lambda;\xi,z)\left(u^2(z)-u_{x,\lambda}^2(z)\right) \mathrm{d}z}\\
& < \int_{B_u^+(x,\lambda)}{\frac{1}{|\xi-z|^2}\left(u^2(z)-u_{x,\lambda}^2(z)\right)\mathrm{d}z}.
\end{aligned}
$$
Using Hardy-Littlewood-Sobolev inequality \ref{lemma Hardy-Littlewood-Sobolev Inequality}, there exists $R>0$, such that
$$
\begin{aligned}
\lVert{P-P_{x,\lambda}}\rVert_{L^{s}(E_P^+(x,\lambda))} &\leqslant R\lVert{u^2-u_{x,\lambda}^2}\rVert_{L^{q}(B_u^+(x,\lambda))} \\
&\leqslant 2R\lVert{u}\rVert_{L^{\infty}(B_u^+(x,\lambda))}\lVert{u-u_{x,\lambda}}\rVert_{L^{q}(B_u^+(x,\lambda))},
\end{aligned}
$$
where $\frac{1}{q} = \frac{1}{s}+\frac{n-2}{n}$. Now choosing $2\varepsilon = n-\frac{n}{q}+1$, we have that there exists $K'>0$, such that
$$
\begin{aligned}
& \qquad \lVert{v-v_{x,\lambda}}\rVert_{L^n(B_v^+(x,\lambda))} \\
& \leqslant K'\lambda^{2\varepsilon}\lVert{(P-P_{x,\lambda})u^2}\rVert_{L^q(E_P^+(x,\lambda))}+K'\lambda^{2\varepsilon}\lVert{P_{x,\lambda}(u^2-u_{x,\lambda}^2)}\rVert_{L^q(E_u^+(x,\lambda))} \\
& \leqslant K'\lambda^{2\varepsilon}\lVert{P-P_{x,\lambda}}\rVert_{L^{s}(E_P^+(x,\lambda))}\lVert{u^2}\rVert_{L^\frac{n}{n-2}(E_P^+(x,\lambda))} \\
& \qquad + K'\lambda^{2\varepsilon}\lVert{u-u_{x,\lambda}}\rVert_{L^{q}(E_u^+(x,\lambda))}\lVert{P_{x,\lambda}(u+u_{x,\lambda})}\rVert_{L^\infty(E_u^+(x,\lambda))} \\
& \leqslant 2K'R\lambda^{2\varepsilon}\lVert{u-u_{x,\lambda}}\rVert_{L^{q}(B_u^+(x,\lambda))}\lVert{u}\rVert_{L^{\infty}(B_u^+(x,\lambda))}\lVert{u^2}\rVert_{L^\frac{n}{n-2}(E_P^+(x,\lambda))} \\
& \qquad + K'\lambda^{2\varepsilon}\lVert{u-u_{x,\lambda}}\rVert_{L^{q}(E_u^+(x,\lambda))}\lVert{P_{x,\lambda}(u+u_{x,\lambda})}\rVert_{L^\infty(E_u^+(x,\lambda))}.
\end{aligned}
$$
If $\lambda$ is bounded above, there exist $c = c(x)>0$, such that $u_{x,\lambda}(\xi)>c$ for $\xi\in B(x,\lambda) \backslash \{x\}$. So, on $E_u^+(x,\lambda)$, we have that
$$
0<P_{x,\lambda}(\xi)[u(\xi)+u_{x,\lambda}(\xi)]\leqslant \frac{P(\xi)u^2(\xi)}{c^2}\left[2u(\xi)\right] .
$$
By the local boundness of $u(\xi)$ and $P(\xi)$, we may choose $\lambda_2$ small enough, such that for $\lambda < \lambda_2$,
$$
\lVert{v-v_{x,\lambda}}\rVert_{L^n(B_v^+(x,\lambda))} \leqslant \frac{1}{10}\lVert{u-u_{x,\lambda}}\rVert_{L^{q}(B_u^+(x,\lambda))}.
$$
Let $\lambda_0 = \min(\lambda_1,\lambda_2)$, so when $\lambda < \lambda_0$,
$$
\lVert{u-u_{x,\lambda}}\rVert_{L^{q}(B_u^+(x,\lambda))} = \lVert{v-v_{x,\lambda}}\rVert_{L^n(B_v^+(x,\lambda))} = 0.
$$
That is, $B_u^+(x,\lambda) = B_v^+(x,\lambda) = \varnothing$. This completes Step 1 for the case $1\leq\alpha\leq n+1$.

\medskip

\textit{Case} (ii) $\alpha\geq n+1$. We will show that, for any $x \in \mathbb{R}^n$, there exists $\lambda_0(x) > 0$ large enough, such that for $\lambda > \lambda_0$ and for any $\xi \in B(x,\lambda) \backslash \{x\}$,
$$
u_{x,\lambda}(\xi)-u(\xi) \leqslant 0\quad\text{and}\quad v_{x,\lambda}(\xi)-v(\xi) \leqslant 0.
$$

Let
$$
\begin{aligned}
B_u^-(x,\lambda) & = \{\xi \in B(x,\lambda) \backslash \{x\} | u_{x,\lambda}(\xi)-u(\xi) > 0\}, \\
B_v^-(x,\lambda) & = \{\xi \in B(x,\lambda) \backslash \{x\} | v_{x,\lambda}(\xi)-v(\xi) > 0\}, \\
E(x,\lambda) & = \{\xi \in B(x,\lambda) \backslash \{x\} | P_{x,\lambda}(\xi)u_{x,\lambda}^2(\xi)-P(\xi)u^2(\xi) > 0\}, \\
E_u^-(x,\lambda) & = E(x,\lambda) \cap B_u^-(x,\lambda), \\
E_P^-(x,\lambda) & = E(x,\lambda) \cap \{\xi \in B(x,\lambda) \backslash \{x\} | P_{x,\lambda}(\xi)-P(\xi) > 0\}.
\end{aligned}
$$

We need to prove the following lemma.
\begin{lem}\label{integral bound}
Suppose $f$ is a nonnegative $C^1$ function such that $\limsup\limits_{|z|\to\infty}f(z)|z|^K$ exists, let
$$
f_{x,\lambda}(\xi) = \frac{\lambda^{L}}{|\xi-x|^{L}}f(\xi^{x,\lambda}),
$$
where $K\geqslant L>0$, then we have that, for $p>0$,
$$
\int_{B(x,\lambda)}f_{x,\lambda}^p(z)\mathrm{d}z \leqslant I_x\lambda^{2n-Lp}
$$
provided that $\lambda>\lambda_x$, where $I_x, \lambda_x$ are constants that depend on $x$ but independent of $\lambda$.
\end{lem}
\begin{proof}
Choose $R>0$, for $|z|>R$, $f(z)|z|^K\leqslant C$. Now choose $\lambda_x = \max\{|x|+R, 3|x|\}$, then for $|z-x|>\lambda$, we have $|z|>R$ and $|z-x|\sim|z|$. So
$$
\begin{aligned}
\int_{B(x,\lambda)}f_{x,\lambda}^p(z)\mathrm{d}z & = \int_{B(x,\lambda)^c}f_{x,\lambda}^p(z^{x,\lambda})\left(\frac{\lambda}{|z-x|}\right)^{2n}\mathrm{d}z\\
& = \int_{B(x,\lambda)^c}f^p(z)\left(\frac{\lambda}{|z-x|}\right)^{2n-pL}\mathrm{d}z\\
& = \int_{B(x,\lambda)^c}[f(z)|z|^K]^p\left(\frac{\lambda}{|z-x|}\right)^{2n-pL}|z|^{-pK}\mathrm{d}z\\
&\leqslant \lambda^{2n-pL}\int_{B(x,\lambda)^c}\frac{C^p}{|z|^{2n+p(K-L)}}\mathrm{d}z \leqslant I_x\lambda^{2n-pL}.
\end{aligned}.
$$
This ends our proof.
\end{proof}

For $\xi \in B_u^-(x,\lambda)$, we have that
$$
\begin{aligned}
0 & < u_{x,\lambda}(\xi)-u(\xi) \\
& \leqslant C_{u,n}\int_{B_v^-(x,\lambda)}{K(x,\lambda;\xi,z)\left(e^{v_{x,\lambda}(z)}-e^{v(z)}\right) \mathrm{d}z}\\
& < C_{u,n}\int_{B_v^-(x,\lambda)}{K(x,\lambda;\xi,z)e^{\eta(z)}(v_{x,\lambda}(z)-v(z))\mathrm{d}z} \\
& < C_{u,n}\int_{B_v^-(x,\lambda)}{\frac{1}{|\xi-z|^{n-1}}e^{\eta(z)}(v_{x,\lambda}(z)-v(z))\mathrm{d}z},
\end{aligned}
$$
where $v(z)<\eta(z)<v_{x,\lambda}(z)$. So according to Hardy-Littlewood-Sobolev Inequality in Lemma \ref{lemma Hardy-Littlewood-Sobolev Inequality}, for some $q \in\left(\frac{n}{n-1},\frac{n}{n-2}\right)$, there exists $K > 0$, such that
$$
\begin{aligned}
\lVert{u_{x,\lambda}-u}\rVert_{L^{q}(B_u^-(x,\lambda))}&\leqslant K\lVert{(v_{x,\lambda}-v)e^{\eta}}\rVert_{L^p(B_v^-(x,\lambda))} \\
& \leqslant K\lVert{e^{\eta}}\rVert_{L^n(B_v^-(x,\lambda))} \lVert{v_{x,\lambda}-v}\rVert_{L^{q}(B_v^-(x,\lambda))} \\
& \leqslant K\left(\int_{B(x,\lambda)}e^{nv_{x,\lambda}(z)}\mathrm{d}z\right)^{\frac{1}{n}} \lVert{v_{x,\lambda}-v}\rVert_{L^{q}(B_v^-(x,\lambda))}\\
& \leqslant K_{n,x}\lambda^{1-n}\lVert{v_{x,\lambda}-v}\rVert_{L^{q}(B_v^-(x,\lambda))}.
\end{aligned}
$$
Here we have used $e^{\eta(z)}<e^{v_{x,\lambda}(z)}$, Lemmas \ref{lemma Asymptotic Property for n = 3}, \ref{lemma Asymptotic Property for n = 4} to deduce $e^{v(z)}\sim |z|^{-\alpha}$, and the Lemma \ref{integral bound} with $f(z) = e^{v(z)}$.

For $\xi \in B_v^-(x,\lambda)$, we have
$$
\begin{aligned}
0 & < v_{x,\lambda}(\xi)-v(\xi)  < C_{v,n}\int_{E(x,\lambda)}{L(x,\lambda;\xi,z)\left(P_{x,\lambda}(z)u_{x,\lambda}^2(z)-P(z)u^2(z)\right)\mathrm{d}z} \\
& < C_{v,n}\int_{E_P^-(x,\lambda)}{L(x,\lambda;\xi,z)u_{x,\lambda}^2(z)[P_{x,\lambda}(z)-P(z)]\mathrm{d}z} \\
& \qquad + C_{v,n}\int_{E_u^-(x,\lambda)}{L(x,\lambda;\xi,z)P(z)[u_{x,\lambda}^2(z)-u^2(z)]\mathrm{d}z} \\
& < C_{v,n}C(\varepsilon)\int_{E_P^-(x,\lambda)}{\frac{\lambda^{2\varepsilon}}{|\xi-z|^{2\varepsilon}}u_{x,\lambda}^2(z)[P_{x,\lambda}(z)-P(z)]\mathrm{d}z} \\
& \qquad + C_{v,n}C(\varepsilon)\int_{E_u^-(x,\lambda)}{\frac{\lambda^{2\varepsilon}}{|\xi-z|^{2\varepsilon}}P(z)[u_{x,\lambda}^2(z)-u^2(z)]\mathrm{d}z}.
\end{aligned}
$$

Let $2\varepsilon = n-1$ and $\frac{1}{s} = 1-\frac{1}{n}-\frac{1}{q}$. Since $\alpha\geqslant n+1$, we have $u^2(z)\sim |z|^{-2(n-1)}$ and $P(z)u(z)\sim |z|^{-(n+1)}$. Thus we deduce that
$$
\begin{aligned}
& \qquad \lVert{v_{x,\lambda}-v}\rVert_{L^{q}(B_v^-(x,\lambda))} \\
& \leqslant K\lambda^{n-1}\lVert{(P_{x,\lambda}-P)u_{x,\lambda}^2}\rVert_{L^p(E_P^-(x,\lambda))}+K\lambda^{n-1}\lVert{P(u_{x,\lambda}^2-u^2)}\rVert_{L^p(E_u^-(x,\lambda))} \\
& \leqslant K\lambda^{n-1}\lVert{P_{x,\lambda}-P}\rVert_{L^n(E_P^-(x,\lambda))}\lVert{u_{x,\lambda}^2}\rVert_{L^q(E_P^-(x,\lambda))} \\
& \qquad + K\lambda^{n-1}\lVert{u_{x,\lambda}-u}\rVert_{L^{q}(E_u^-(x,\lambda))}\lVert{P(u+u_{x,\lambda})}\rVert_{L^n(E_u^-(x,\lambda))} \\
& \leqslant K\lambda^{n-1}\lVert{u_{x,\lambda}^2-u^2}\rVert_{L^{\frac{n}{n-1}}(E_u^-(x,\lambda))}\left(\int_{B(x,\lambda)}u_{x,\lambda}^{2q}(z)\mathrm{d}z\right)^{\frac{1}{q}}  \\
& \qquad + 2K\lambda^{n-1}\lVert{u_{x,\lambda}-u}\rVert_{L^{q}(E_u^-(x,\lambda))}\left(\int_{B(x,\lambda)}[P_{x,\lambda}(z)u_{x,\lambda}(z)]^n\mathrm{d}z\right)^{\frac{1}{n}}  \\
& \leqslant K\lambda^{\frac{2n}{q}-(n-1)}\lVert{u_{x,\lambda}-u}\rVert_{L^q(E_u^-(x,\lambda))}\left(\int_{B(x,\lambda)}u_{x,\lambda}^{s}(z)\mathrm{d}z\right)^{\frac{1}{s}}+K\lVert{u_{x,\lambda}-u}\rVert_{L^{q}(E_u^-(x,\lambda))} \\
& \leqslant 2K\lVert{u_{x,\lambda}-u}\rVert_{L^{q}(E_u^-(x,\lambda))}+K\lVert{u_{x,\lambda}-u}\rVert_{L^{q}(E_u^-(x,\lambda))}.
\end{aligned}
$$
Here we have used Lemmas \ref{integral bound}, \ref{lemma Asymptotic Property for n = 3} and \ref{lemma Asymptotic Property for n = 4} again.

Now we may choose $\lambda_0$ big enough, such that, for $\lambda > \lambda_0$,
$$
\lVert{v_{x,\lambda}-v}\rVert_{L^{q}(B_v^-(x,\lambda))} \leqslant \frac{1}{10}\lVert{v_{x,\lambda}-v}\rVert_{L^{q}(B_u^-(x,\lambda))}.
$$
That gives
$$
\lVert{u_{x,\lambda}-u}\rVert_{L^{q}(B_u^-(x,\lambda))} = \lVert{v_{x,\lambda}-v}\rVert_{L^{q}(B_v^-(x,\lambda))} = 0.
$$
That is, $B_u^-(x,\lambda) = B_v^-(x,\lambda) = \varnothing$. This completes Step 1 for the case $\alpha\geq n+1$.

\bigskip

Step 2. Moving the sphere $S_{\lambda}(x)$ outward or inward until the limiting position.

\medskip

In what follows, we will derive contradictions in both the cases $\alpha>n+1$ and $1\leq \alpha<n+1$, and hence we must have $\alpha=n+1$.

\medskip

\textit{Case} (i) $1\leq\alpha<n+1$. Let us define
\begin{equation}\label{Definition of Supremum Limiting Radius}
\overline{\lambda}(x) = \sup\left\{\mu > 0 | (u-u_{x,\lambda})(\xi) \leqslant 0 \;\text{and}\; (v-v_{x,\lambda})(\xi) \leqslant 0, \forall 0 < \lambda < \mu , \xi \in B(x,\lambda) \backslash \{x\}\right\}.
\end{equation}
We will prove that, for any $\overline{x}\in \mathbb{R}^n$, $\overline{\lambda}(\overline{x}) = \infty$.

By the definition of $\overline{\lambda}(\overline{x})$ and continuity, we have
$$
u(\xi) \leqslant u_{\overline{x},\overline{\lambda}(\overline{x})}(\xi) \,,\, v(\xi) \leqslant v_{\overline{x},\overline{\lambda}(\overline{x})}(\xi), \qquad \forall \xi \in B(\overline{x},\overline{\lambda}(\overline{x})).
$$
By the positivity of kernel $M(x,\lambda;\xi,z)$, we obtain that for $\xi \in B(\overline{x},\overline{\lambda}(\overline{x}))$, $P(\xi)\leqslant P_{\overline{x},\overline{\lambda}(\overline{x})}(\xi)$.
So for $\xi \in B(\overline{x},\overline{\lambda}(\overline{x}))$, by the positivity of kernel $L(x,\lambda;\xi,z)$, we have that
$$
\begin{aligned}
v(\xi)-v_{\overline{x},\lambda}(\xi) & = \left[\alpha-n-1\right]\ln\left(\frac{\lambda}{|\xi-\overline{x}|}\right)\\
&\qquad+C_{v,n}\int_{B(\overline{x},\lambda)}{L(\overline{x},\lambda;\xi,z)\left(P(z)u^2(z)-P_{\overline{x},\lambda}(z)u_{\overline{x},\lambda}^2(z)\right)\mathrm{d}z}\\
&\leqslant[\alpha-n-1]\ln\left(\frac{\lambda}{|\xi-\overline{x}|}\right)<0.
\end{aligned}
$$
Then again by the positivity of Kernel $K(x,\lambda;\xi,z)$ as well as $C^1$ continuity of $v$, we have, for $\xi \in B(\overline{x},\overline{\lambda}(\overline{x}))$,
$$
u(\xi)-u_{\overline{x},\lambda}(\xi) = C_{u,n}\int_{B(\overline{x},\lambda)}{K(\overline{x},\lambda;\xi,z)\left(e^{v(z)}-e^{v_{\overline{x},\lambda}(z)}\right)\mathrm{d}z}<0.
$$
Then for small $\delta$, $u(\xi)-u_{\overline{x},\overline{\lambda}(\overline{x})}(\xi)$ is a $C^1$ negative function defined on compact subset $\overline{B}(\overline{x},\overline{\lambda}(\overline{x})-\delta)\backslash B(\overline{x},\delta)$, so is $v(\xi)-v_{\overline{x},\overline{\lambda}(\overline{x})}(\xi)$. As a consequence, they must attain their negative maximums, that is, there exists $C_0(\delta) > 0$ such that
$$
\left\{
\begin{aligned}
u(\xi)-u_{\overline{x},\overline{\lambda}(\overline{x})}(\xi)  & \leqslant -C_0, \\
v(\xi)-v_{\overline{x},\overline{\lambda}(\overline{x})}(\xi)  & \leqslant -C_0,
\end{aligned}
\right.
\qquad
\forall \xi \in \overline{B}(\overline{x},\overline{\lambda}(\overline{x})-\delta)\backslash B(\overline{x},\delta).
$$
So by the uniform continuity w.r.t. $\lambda$ of $u_{x,\lambda},v_{x,\lambda}$ in compact subset, there exists sufficiently small $\varepsilon(\delta) > 0$, such that for all $\lambda \in (\overline{\lambda}(\overline{x}),\overline{\lambda}(\overline{x})+\varepsilon)$, we have
$$
\left\{
\begin{aligned}
u(\xi) - u_{\overline{x},\lambda}(\xi) &<-\frac{C_0}{2}, \\
v(\xi) - v_{\overline{x},\lambda}(\xi) &<-\frac{C_0}{2},
\end{aligned}
\right.
\qquad
\forall \xi \in \overline{B}(\overline{x},\overline{\lambda}(\overline{x})-\delta)\backslash B(\overline{x},\delta).
$$
For $\lambda \in (\overline{\lambda}(\overline{x}),\overline{\lambda}(\overline{x})+\varepsilon)$, we define
$$
A(\overline{x},\lambda,\varepsilon,\delta) \triangleq \{\xi \in B(\overline{x},\lambda) | |\xi-\overline{x}| < \delta \;\text{or}\; \overline{\lambda}(\overline{x})-\delta < |\xi-\overline{x}| < \lambda\}.
$$
Thus $B_u^+(\overline{x},\lambda) = B_v^+(\overline{x},\lambda)\subset A(\overline{x},\lambda,\varepsilon,\delta)$. We deduced that there exists $A = A(\overline{\lambda},\overline{x})>0$ such that
$$
|A(\overline{x},\lambda,\varepsilon,\delta)|\leqslant A(\varepsilon+\delta)^{n-1}.
$$
Denote $\overline{\lambda}(\overline{x})$ by $\overline{\lambda}$, from the above analysis, we deduced that
\begin{equation}\label{e1}
\begin{aligned}
\lVert{u-u_{\overline{x},\lambda}}\rVert_{L^{q}(B_u^+(\overline{x},\lambda))}
& \leqslant K\lVert{e^{v}}\rVert_{L^q(B_v^+(\overline{x},\lambda))} \lVert{v-v_{\overline{x},\lambda}}\rVert_{L^n(B_v^+(\overline{x},\lambda))},
\end{aligned}
\end{equation}
and that
\begin{equation}\label{e2}
\begin{aligned}
& \quad \lVert{v-v_{\overline{x},\lambda}}\rVert_{L^n(B_v^+(\overline{x},\lambda))} \\
& \leqslant K'\lambda^{2\varepsilon_0}\lVert{u}\rVert_{L^{\infty}(B_u^+(\overline{x},\lambda))}\lVert{u-u_{\overline{x},\lambda}}\rVert_{L^{q}(B_u^+(\overline{x},\lambda))}\lVert{u^2}\rVert_{L^\frac{n}{n-2}(E_P^+(\overline{x},\lambda))} \\
& \qquad + K'\lambda^{2\varepsilon_0}\lVert{u-u_{\overline{x},\lambda}}\rVert_{L^{q}(E_u^+(\overline{x},\lambda))}\lVert{P_{x,\lambda}(u+u_{\overline{x},\lambda})}\rVert_{L^\infty(E_u^+(\overline{x},\lambda))}.
\end{aligned}
\end{equation}
By the the local boundedness of $P,u,v$, we have
$$
\begin{aligned}
\lVert{e^{v}}\rVert_{L^q(B_v^+(x,\lambda))}
&\leqslant\lVert{e^{v}}\rVert_{L^q(A(\overline{x},\lambda,\varepsilon,\delta))} \\
&\leqslant  2\lVert{e^v}\rVert_{L^\infty(B(\overline{x},\overline{\lambda}+1))}|A(\overline{x},\lambda,\varepsilon,\delta)|^{\frac{1}{q}} \\
&\leqslant C(\delta+\varepsilon)^{\frac{n-1}{q}},
\end{aligned}
$$
and
$$
\lambda^{2\varepsilon_0}\lVert{u}\rVert_{L^{\infty}(B_u^+(\overline{x},\lambda))}\leqslant (\overline{\lambda}+1)^{2\varepsilon_0}\lVert{u}\rVert_{L^{\infty}(B_u^+(\overline{x},\overline{\lambda}+1))}\leqslant C_1,
$$
$$
\lambda^{2\varepsilon_0}\lVert{P_{x,\lambda}(u+u_{x,\lambda})}\rVert_{L^\infty(E_u^+(x,\lambda))}\leqslant 2(\overline{\lambda}+1)^{2\varepsilon_0}\lVert{Pu}\rVert_{L^\infty(B(\overline{x},\overline{\lambda}+1))}\leqslant C_2.
$$
Thus we can choose $\delta,\varepsilon$ small enough and deduce from \eqref{e1} and \eqref{e2} that $B_u^+(x,\lambda) = B_v^+(x,\lambda) = \varnothing$ for $\lambda \in (\overline{\lambda}(\overline{x}),\overline{\lambda}(\overline{x})+\varepsilon)$, which contradicts the definition of $\overline{\lambda}(\overline{x})$. So we must have $\overline{\lambda}(\overline{x}) = \infty$ in Case (i) $1\leq \alpha<n+1$.

From the conclusion (ii) in Lemma \ref{Calculus Lemma} and the finite mass condition \eqref{fm}, we must have $u\equiv0$, which will lead to a contradiction again by the first equation in system \eqref{system}. Hence, Case (i) $1\leq \alpha<n+1$ can not happen.

\medskip

\textit{Case} (ii) $\alpha> n+1$. Let us define
\begin{equation}\label{Definition of Infimum Limiting Radius}
\overline{\lambda}(x) = \inf\left\{\mu > 0 | (u_{x,\lambda}-u)(\xi) \leqslant 0 \;\text{and}\; (v_{x,\lambda}-v)(\xi) \leqslant 0, \forall 0 < \lambda < \mu , \xi \in B(x,\lambda) \backslash \{x\}\right\}.
\end{equation}
We will show that, for any $\overline{x}\in \mathbb{R}^n$, $\overline{\lambda}(\overline{x}) = 0$.

By the definition of $\overline{\lambda}(\overline{x})$ and continuity, we have
$$
u(\xi) \geqslant u_{\overline{x},\overline{\lambda}(\overline{x})}(\xi), \,\quad \, v(\xi) \geqslant v_{\overline{x},\overline{\lambda}(\overline{x})}(\xi), \qquad \forall \xi \in B(\overline{x},\overline{\lambda}(\overline{x})).
$$
By the positivity of kernel $M(x,\lambda;\xi,z)$, we obtain that for $\xi \in B(\overline{x},\overline{\lambda}(\overline{x}))$, $P(\xi)\geqslant P_{\overline{x},\overline{\lambda}(\overline{x})}(\xi)$. Thus by the positivity of kernel $L(x,\lambda;\xi,z)$, we have that, for $\xi \in B(\overline{x},\overline{\lambda}(\overline{x}))$,
$$
\begin{aligned}
v_{\overline{x},\lambda}(\xi)-v(\xi) & = \left[n+1-\alpha\right]\ln\left(\frac{\lambda}{|\xi-\overline{x}|}\right) \\
&\qquad+C_{v,n}\int_{B(\overline{x},\lambda)}{L(\overline{x},\lambda;\xi,z)\left(P_{\overline{x},\lambda}(z)u_{\overline{x},\lambda}^2(z)-P(z)u^2(z)\right)\mathrm{d}z} \\
&\leqslant[n+1-\alpha]\ln\left(\frac{\lambda}{|\xi-\overline{x}|}\right)<0.
\end{aligned}
$$
Then again by the positivity of Kernel $K(x,\lambda;\xi,z)$, we have, for $\xi \in B(\overline{x},\overline{\lambda}(\overline{x}))$,
$$
u_{\overline{x},\lambda}(\xi)-u(\xi) = C_{u,n}\int_{B(\overline{x},\lambda)}{K(x,\lambda;\xi,z)\left(e^{v_{\overline{x},\lambda}(z)}-e^{v(z)}\right)\mathrm{d}z}<0.
$$
Then for small $\delta$, $u_{\overline{x},\overline{\lambda}(\overline{x})}(\xi)-u(\xi)$ is a $C^1$ negative function defined on compact subset $\overline{B}(\overline{x},\overline{\lambda}(\overline{x})-\delta)\backslash B(\overline{x},\delta)$, so is $v_{\overline{x},\overline{\lambda}(\overline{x})}(\xi)-v(\xi)$. As a consequence, they must attain their negative maximum, that is, there exists $C_0(\delta) > 0$ such that
$$
\left\{
\begin{aligned}
u_{\overline{x},\overline{\lambda}(\overline{x})}(\xi)-u(\xi)  & \leqslant -C_0, \\
v_{\overline{x},\overline{\lambda}(\overline{x})}(\xi)-v(\xi)  & \leqslant -C_0,
\end{aligned}
\right.
\qquad
\forall \xi \in \overline{B}(\overline{x},\overline{\lambda}(\overline{x})-\delta)\backslash B(\overline{x},\delta).
$$
So by the uniform continuity w.r.t. $\lambda$ of $u_{x,\lambda},v_{x,\lambda}$ in compact subset, there exists sufficiently small $\varepsilon(\delta) > 0$, such that for all $\lambda \in (\overline{\lambda}(\overline{x}),\overline{\lambda}(\overline{x})-\varepsilon)$, we have
$$
\left\{
\begin{aligned}
u_{\overline{x},\lambda}(\xi) - u(\xi) &<-\frac{C_0}{2}, \\
v_{\overline{x},\lambda}(\xi) - v(\xi) &<-\frac{C_0}{2},
\end{aligned}
\right.
\qquad
\forall \xi \in \overline{B}(\overline{x},\overline{\lambda}(\overline{x})-\delta)\backslash B(\overline{x},\delta).
$$
Thus $B_u^-(\overline{x},\lambda) = B_v^-(\overline{x},\lambda)\subset A(\overline{x},\lambda,\varepsilon,\delta)$. For the sake of simplicity, we denote $\overline{\lambda}(\overline{x})$ by $\overline{\lambda}$.

\medskip

Now we prove the following Lemma.
\begin{lem}\label{annulus integral bound}
Suppose $f$ is a nonnegative $C^1$ function such that $\limsup\limits_{|z|\to\infty}f(z)|z|^K$ exists, then let
$$
f_{\overline{x},\lambda}(\xi) = \frac{\lambda^{L}}{|\xi-\overline{x}|^{L}}f(\xi^{\overline{x},\lambda}),
$$
where $K\geqslant L>0$, then we have that, for $p>0$ and $\lambda \in (\overline{\lambda},\overline{\lambda}-\varepsilon)$,
$$
\int_{A(\overline{x},\lambda,\varepsilon,\delta)}f_{\overline{x},\lambda}^p(z)\mathrm{d}z \leqslant I_{\overline{x}} |A(\overline{x},\lambda,\varepsilon,\delta)|
$$
for $\delta<\delta_{\overline{x}}$ and $\varepsilon<\varepsilon_{\overline{x}}$, where $I_{\overline{x}}, \delta_{\overline{x}}$ are constants that depend on $\overline{x}$ but independent of $\lambda$.
\end{lem}
\begin{proof}
Similar to Lemma \ref{integral bound}, we have
$$
\begin{aligned}
\int_{B(\overline{x},\delta)}f_{\overline{x},\lambda}^p(z)\mathrm{d}z
& = \int_{B(\overline{x},\lambda^2\delta^{-1})^c}[f(z)|z|^K]^p\left(\frac{\lambda}{|z-\overline{x}|}\right)^{2n-pL}|z|^{-pK}\mathrm{d}z \\
&\leqslant \lambda^{2n-pL}\int_{B(\overline{x},\lambda^2\delta^{-1})^c}\frac{C^p}{|z|^{2n+p(K-L)}}\mathrm{d}z \\
& = C_p\left(\frac{\delta}{\lambda^2}\right)^{n+p(K-L)}\lambda^{2n-pL}\\
&\leqslant C_p(\overline{\lambda}+1)^{10n+pK}\delta^n.
\end{aligned}
$$
Choose $\delta_{\overline{x}}\leqslant\frac{\overline{\lambda}}{2}$, and set $\lambda_s = 2(\overline{\lambda}+1)^2(\overline{\lambda})^{-1}$, we get
$$
\begin{aligned}
\int_{B(\overline{x},\lambda)\backslash B(\overline{x},\overline{\lambda}-\delta)}f_{\overline{x},\lambda}^p(z)\mathrm{d}z
&= \int_{B(\overline{x},\lambda^2(\overline{\lambda}-\delta)^{-1})\backslash B(\overline{x},\lambda)}[f(z)]^p\left(\frac{\lambda}{|z-\overline{x}|}\right)^{2n-pL}\mathrm{d}z \\
&\leqslant C_p\lVert{f^p}\rVert_{L^\infty(B(\overline{x},\lambda_s))}\lambda_s^{2n-pL}(\delta+\varepsilon)^{n-1}.
\end{aligned}
$$
Combining the above two inequalities together, we get the desired estimate in Lemma \ref{annulus integral bound}.
\end{proof}

From the above analysis, we deduced that
\begin{equation}\label{e3}
\begin{aligned}
\lVert{u_{\overline{x},\lambda}-u}\rVert_{L^{q}(B_u^-(\overline{x},\lambda))}
& \leqslant K\lVert{e^{v_{\overline{x},\lambda}}}\rVert_{L^n(B_v^-(\overline{x},\lambda))} \lVert{v_{\overline{x},\lambda}-v}\rVert_{L^q(B_v^-(\overline{x},\lambda))},
\end{aligned}
\end{equation}
and that
\begin{equation}\label{e4}
\begin{aligned}
& \quad \lVert{v_{\overline{x},\lambda}-v}\rVert_{L^q(B_v^-(\overline{x},\lambda))} \\
& \leqslant
K\lambda^{n-1}\lVert{u_{\overline{x},\lambda}-u}\rVert_{L^{q}(E_u^-(\overline{x},\lambda))}\lVert{u_{\overline{x},\lambda}}\rVert_{L^s(E_u^-(\overline{x},\lambda))}\lVert{u_{\overline{x},\lambda}^2}\rVert_{L^q(E_P^-(\overline{x},\lambda))} \\
& \qquad + K\lambda^{n-1}\lVert{u_{\overline{x},\lambda}-u}\rVert_{L^{q}(E_u^-(\overline{x},\lambda))}\lVert{P(u+u_{\overline{x},\lambda})}\rVert_{L^n(E_u^-(\overline{x},\lambda))}.
\end{aligned}
\end{equation}

By Lemma \ref{annulus integral bound}, we can choose $\delta,\varepsilon$ small enough and deduce from \eqref{e3} and \eqref{e4} that $B_u^-(x,\lambda) = B_v^-(x,\lambda) = \varnothing$ for $\lambda \in (\overline{\lambda}(\overline{x}),\overline{\lambda}(\overline{x})-\varepsilon)$. This contradicts the definition of $\overline{\lambda}(\overline{x})$. Thus we must have $\overline{\lambda}(\overline{x}) = 0$ in Case (ii) $\alpha>n+1$.

From the conclusion (ii) in Lemma \ref{Calculus Lemma} (replacing $u$ by $-u$ therein), we deduce that $u\equiv C$ for some constant $C$. Due to the finite mass condition \ref{fm}, we must have $u\equiv0$. However, by the first equation in the system \eqref{system}, we have
$$0=e^{v(x)}>0 \qquad \text{in} \,\, \mathbb{R}^{n}.$$
This is a contradiction and hence Case (ii) $\alpha>n+1$ is impossible.

\bigskip

From the contradictions derived in both Cases (i) and (ii), we conclude that
\begin{equation}\label{a8}
  \alpha:=C_{n}\int_{\mathbb{R}^n}{P_n(x)\mathrm{d}x}=n+1,
\end{equation}
where $C_{n}=\frac{1}{2\pi^{2}}$ if $n=3$ and $C_{n}=\frac{1}{8\pi^{2}}$ if $n=4$. In this case, we deduce from Step 1 that, for $\lambda>0$ large,
 \begin{equation}\label{a9}
 	u_{x,\lambda}(\xi)-u(\xi)\leq 0 \quad \text{and} \quad v_{x,\lambda}(\xi)-v(\xi)\leq 0, \qquad \forall \xi \in B(x,\lambda)\setminus\{x\};
 \end{equation}
while for $\lambda>0$ small,
\begin{equation}\label{a10}
	u_{x,\lambda}(\xi)-u(\xi)\geq 0 \quad \text{and} \quad v_{x,\lambda}(\xi)-v(\xi)\geq 0, \qquad \forall \xi \in B(x,\lambda)\setminus\{x\}.
\end{equation}
If the critical scale (defined in \eqref{Definition of Supremum Limiting Radius}) $\overline{\lambda}(x)<+\infty$, then we must have $u_{x,\overline{\lambda}(x)}=u$ and $v_{x,\overline{\lambda}(x)}=v$ in $B(x,\overline{\lambda}(x))\setminus\{x\}$, or else the sphere $S_{\lambda}$ can be moved a bit further such that \eqref{a10} still hold (see \emph{Case (ii)} or \emph{Case (i)} in Step 2), which contradicts the definition \eqref{Definition of Supremum Limiting Radius} of $\overline{\lambda}(x)$. If the critical scale (defined in \eqref{Definition of Supremum Limiting Radius}) $\overline{\lambda}(x)=+\infty$, it follows from \eqref{a9} that $u_{x,\lambda}=u$ and $v_{x,\lambda}=v$ in $B(x,\lambda)\setminus\{x\}$ for $\lambda$ sufficiently large. As a consequence, for arbitrary $x\in\mathbb R^n$, there exists a $\lambda>0$ (depending on $x$) such that
\begin{equation}\label{m42}
	u_{x,\lambda}(\xi)=u(\xi) \quad \text{and} \quad v_{x,\lambda}(\xi)=v(\xi), \qquad \forall \xi \in B(x,\lambda)\setminus\{x\}.
\end{equation}
Then, we infer from the conclusion (i) in Lemma \ref{Calculus Lemma} that, for some $C\in\mathbb{R}$, $\mu>0$ and $x_0\in \mathbb{R}^n$, $u$ must be of the form
\begin{equation}\label{m43}
	u(x)=C{\left(\frac{\mu}{1+\mu^{2}|x-x_0|^2}\right)}^{\frac{n-1}{2}}, \qquad \forall \, x\in\mathbb{R}^{n},
\end{equation}
which combined with the first equation in system \eqref{system} and the asymptotic behavior of $v$ in Lemmas \ref{lemma Integral Representation of v for n = 3} and \ref{lemma Integral Representation of v for n = 3} imply
\begin{equation}\label{m44}
	v(x)=\frac{n+1}{2}\ln\left[\frac{C'\mu}{1+\mu^{2}|x-x_0|^2}\right], \qquad \forall \, x\in\mathbb{R}^{2}.
\end{equation}
Then, by the formula \eqref{a8} and direct calculations, we can obtain that $C=\frac{2\cdot2^{\frac{1}{4}}}{\sqrt{\pi}}$ if $n=3$, $C=\frac{2\cdot 30^{\frac{1}{4}}}{\sqrt{\pi}}$ if $n=4$, and hence
\begin{equation}\label{e5}
u(x)=\frac{2\cdot2^{\frac{1}{4}}\mu}{\sqrt{\pi}\left(1+\mu^{2}|x-x_{0}|^{2}\right)},\qquad \text{if} \,\, n=3,
\end{equation}
\begin{align}\label{e6}
u(x)=\frac{2\cdot30^{\frac{1}{4}}}{\sqrt{\pi}}\left(\frac{\mu}{1+\mu^{2}|x-x_{0}|^{2}}\right)^
\frac{3}{2}, \qquad \text{if} \,\, n=4.
\end{align}
Moreover, by the asymptotic behavior of $u$ in Lemmas \ref{lemma Asymptotic Property for n = 3} and \ref{lemma Asymptotic Property for n = 4}, one has
$$
\beta:= \frac{1}{2\pi^2}\int_{\mathbb{R}^4}{e^{pv(x)}\mathrm{d}x}=\frac{2\cdot2^{\frac{1}{4}}}{\sqrt{\pi}}\frac{1}{\mu}, \qquad \text{if} \,\, n=3,
$$
$$
\beta:= \frac{1}{4\pi^2}\int_{\mathbb{R}^4}{e^{pv(x)}\mathrm{d}x}=\frac{2\cdot30^{\frac{1}{4}}}{\sqrt{\pi}}\frac{1}{\mu^{\frac{3}{2}}}, \qquad \text{if} \,\, n=4.
$$
Combining this with \eqref{m44} yield that $C'=\frac{2}{\sqrt[4]{\pi}}2^{\frac{1}{8}}$ if $n=3$, $C'=\frac{\sqrt{6}}{\sqrt[5]{\pi}}5^{\frac{1}{10}}$ if $n=4$, and hence
\begin{equation}\label{e7}
v(x)=2\ln\left(\frac{\frac{2}{\sqrt[4]{\pi}}2^{\frac{1}{8}}\mu}
{1+\mu^{2}|x-x_{0}|^{2}}\right), \qquad \text{if} \,\, n=3,
\end{equation}
\begin{align}\label{e8}
v(x)=\frac{5}{2}\ln\left(\frac{\frac{\sqrt{6}}{\sqrt[5]{\pi}}5^{\frac{1}{10}}\mu}
{1+\mu^{2}|x-x_{0}|^{2}}\right), \qquad \text{if} \,\, n=4.
\end{align}
In addition, one can verify by calculations that $(u,v)$ given by \eqref{e5}, \eqref{e6}, \eqref{e7} and \eqref{e8} is indeed a pair of solutions to the PDE system \eqref{system}. This concludes our proof of Theorem \ref{theorem Main Result}. \qed

\end{document}